\documentclass[10pt]{amsart}
\usepackage{latexsym, amssymb, amsfonts, amscd}

\theoremstyle{plain}
\newtheorem{Thm}{Theorem}[subsection]
\newtheorem{Cor}[Thm]{Corollary}

\newtheorem{Prop}[Thm]{Proposition}
\newtheorem{Lem}[Thm]{Lemma}

\theoremstyle{definition}
\newtheorem{Def}[Thm]{Definition}

\newtheorem{Emp}[Thm]{}

\numberwithin{equation}{section}

\newcommand{\La}{\Lambda}

\newcommand{\B}[1]{\mathbb#1}
\newcommand{\cal}[1]{\mathcal{#1}}
\newcommand{\C}[1]{\cal#1}

\newcommand{\form}[1]{(\ref{Eq:#1})}

\newcommand{\lra}{\longrightarrow}

\newcommand{\hra}{\hookrightarrow}
\newcommand{\wt}{\widetilde}
\newcommand{\wh}{\widehat}
\newcommand{\Gm}{\Gamma}

\newcommand{\gm}{\gamma}

\newcommand{\dt}{\delta}
\newcommand{\Dt}{\Delta}
\newcommand{\bs}{\backslash}

\newcommand{\al}{\alpha}

\newcommand{\rs}[1]{Section \ref{S:#1}}
\newcommand{\rl}[1]{Lemma \ref{L:#1}}

\newcommand{\rp}[1]{Proposition \ref{P:#1}}

\newcommand{\re}[1]{\ref{E:#1}}
\newcommand{\rco}[1]{Corollary \ref{C:#1}}

\newcommand{\rt}[1] {Theorem \ref{T:#1}}
\newcommand{\rd}[1]{Definition \ref{D:#1}}
\newcommand{\sm}{\smallsetminus}

\newcommand{\Map}{\operatorname{Map}}

\newcommand{\id}{\operatorname{id}}
\newcommand{\pr}{\operatorname{pr}}
\newcommand{\ev}{\operatorname{ev}}

\newcommand{\im}{\operatorname{Im}}

\newcommand{\Fun}{\operatorname{Fun}}

\newcommand{\Mor}{\operatorname{Mor}}
\newcommand{\map}{\operatorname{map}}
\newcommand{\Ho}{\operatorname{Ho}}

\newcommand{\pt}{\operatorname{pt}}

\newcommand{\Ob}{\operatorname{Ob}}

\newcommand{\p}{\partial}

\newcommand{\Id}{\operatorname{Id}}

\newcommand{\colim}{\operatorname{colim}}
\newcommand{\Hom}{\operatorname{Hom}}

\renewcommand{\S}{\frak{S}}
\renewcommand{\P}{\frak{P}}
\begin{document}

\title[Yoneda lemma for complete Segal spaces]%
{Yoneda lemma for complete Segal spaces}

\author{David Kazhdan}
\author{Yakov Varshavsky}
\address{Institute of Mathematics\\
The Hebrew University of Jerusalem\\
Givat-Ram, Jerusalem,  91904\\
Israel} 
\email{kazhdan@math.huji.ac.il, vyakov@math.huji.ac.il}



\thanks{D.K. was partially supported by the ERC grant No. 247049-GLC,
Y.V. was partially supported by the ISF grant No. 1017/13}
\begin{abstract} In this note we formulate and give a self-contained proof of
the Yoneda lemma for $\infty$-categories in the language of
complete Segal spaces.
\end{abstract}
\maketitle \centerline{\em To the memory of I.M.Gelfand}
\tableofcontents
\section*{Introduction}
In recent years $\infty$-categories or, more formally,
$(\infty,1)$-categories appear in various areas of mathematics.
For example, they became a necessary ingredient in the geometric
Langlands problem.  In his books \cite{Lu,Lu2} Lurie developed a
theory of $\infty$-categories in the language of quasi-categories
and extended many results of the ordinary category theory to this
setting.

In his work \cite{Re} Rezk introduced another model of
$\infty$-categories, which he called complete Segal spaces. This
model has certain advantages. For example, it has a 
generalization to $(\infty,n)$-categories (see \cite{Re2}).

It is natural to extend results of the ordinary category
theory to the setting of complete Segal spaces. In this note we do
this for the Yoneda lemma.

To formulate it, we need to construct a convenient model of the
"$\infty$-category of spaces", which an $\infty$-analog of the
category of sets. Motivated by Lurie's results,  
we define this $\infty$-category to be the simplicial space "classifying left
fibrations". After this is done, the construction of the Yoneda embedding and 
the proof of the Yoneda lemma goes almost like in the case of ordinary categories. 

In our next works \cite{KV1,KV2} we study adjoint functors, limits and
colimits, show a stronger version of the Yoneda lemma, and generalize results of this paper 
to the setting of $(\infty,n)$-categories.

We thank Emmanuel Farjoun, Vladimir Hinich and Nick Rozenblyum for
stimulating conversations and valuable remarks.

This paper is organized as follows. To make the work self-contained,
in the first section we introduce basic definitions and discuss
properties of model categories, simplicial sets, simplicial spaces
and Segal spaces, assuming only basic category theory. In the
second section we introduce left fibrations, construct the
$\infty$-category of spaces $\S$, and formulate and prove the
Yoneda lemma. Next, in the third section, we study quasifibrations
of simplicial spaces, which are needed for our argument and are also
very interesting objects for their own. Finally, the last section
is devoted to the proof of the properties of $\S$, formulated in
the second section. 

\section{Preliminaries}

\subsection{Model categories}

\begin{Emp} \label{E:rlp}
{\bf Notation.} Let $\C{C}$ be a category. (a) For an element $Z\in\C{C}$, we denote by 
$\C{C}/Z$ the overcategory over $Z$. 

(b) For a pair of
morphisms $i:A\to B$ and $p:X\to Y$ in $\C{C}$, we denote by
$\Hom_{\C{C}}(i,p)$ the set of commutative diagrams in ${\C{C}}$
\begin{equation} \label{Eq:cd}
\begin{CD}
A @>a>> X\\
@ViVV @VpVV\\
B @>b>> Y.
\end{CD}
\end{equation}
We say that $i$ is a {\em retract} of $p$, if there exist
$\al\in\Hom_{\C{C}}(i,p)$ and $\beta\in\Hom_{\C{C}}(p,i)$ such
that $\beta\circ\al=\Id_i$.

(c) We say that $p$ has the {\em right lifting property (RLP)}
with respect  to $i$ (and that $i$ has the {\em left lifting
property (LLP)} with respect  to $p$), if for every commutative
diagram \form{cd} there exists a morphism $c:B\to X$ such that
$p\circ c=b$ and $c\circ i=a$.

Equivalently, this happens if and only if the natural map of sets
\[
(i^*,p_*):\Hom(B,X)\to \Hom(A,X)\times_{\Hom(A,Y)}\Hom(B,Y)
\]
is surjective.

(d) Assume that $\C{C}$ has fiber products. Then for every
morphism $f:X\to Y$ in $\C{C}/Z$ and morphism
$g:Z'\to Z$ in $\C{C}$, we write $g^*(f):g^*(X)\to g^*(Y)$ instead of
$f\times_Z Z':X\times_Z Z'\to Y\times_Z Z'$ and call it {\em the
pullback of $f$}.

(e) We say that a category ${\C{C}}$ is {\em Cartesian}, if
${\C{C}}$ has finite products, and for every $X,Y\in {\C{C}}$
there exists an element $X^Y\in {\C{C}}$, representing a
functor $Z\mapsto\Hom(Z\times Y,X)$, which is called {\em the internal hom} of $X$ and $Y$. 
\end{Emp}

\begin{Emp} \label{E:expsh}
{\bf Example.} Let $\C{C}$ be the category of functors
$\C{C}=\Fun(\Gm^{op},Set)$, where $\Gm$ is a small category, and
$Set$ is the category of sets.

(a) Since category $Set$ has all limits and colimits,  category
$\C{C}$ also has these properties. Explicitly, every functor
$\al:I\to\C{C}$ defines a functor $\al(\gm):I\to Set$ for each
$\gm\in\Gm$, and we have $\lim_I(\al)(\gm)=\lim_I(\al(\gm))$ and
similarly for colimits. In particular, category $\C{C}$ has
products.

(b) For every $\gm\in\Gm$, denote by $F_{\gm}\in\C{C}$ the
representable functor $\Hom_{\Gm}(\cdot,\gm)$. Then for every
$X,Y\in\C{C}$ there exists their internal hom $X^Y\in\C{C}$, defined
by the rule $X^Y(\gm)=\Hom(Y\times F_{\gm},X)$ with obvious
transition maps. In other words, category $\C{C}$ is Cartesian.
\end{Emp}

The following lemma is straightforward.

\begin{Lem} \label{L:RLP}
Let ${\C{C}}$ be a Cartesian category, and let $i:A\to B$,
$j:A'\to B'$ and $p:X\to Y$ be morphisms in ${\C{C}}$. Then $j$
has the LLP with respect to $(i^*,p_*):X^B\to X^A\times_{Y^A}Y^B$
if and only if $(i_*,j_*):(A\times B')\sqcup_{(A\times
A')}(B\times A')\to B\times B'$ has the LLP with respect to $p$.
\end{Lem}

\begin{Def} \label{D:model}
(compare \cite[II,1]{GJ}). A {\em model category} is a category
${\C{C}}$, equipped with three collections of morphisms, called
cofibrations, fibrations and weak equivalences, which satisfy the
following axioms:

CM1: The category ${\C{C}}$ has all finite limits and colimits.

CM2 (2-out-of-3): In a diagram
$X\overset{f}{\to}Y\overset{g}{\to}Z$ if any two of the morphisms
$f,g$ and $g\circ f$ are weak equivalences, then so is the third.

CM3 (retract): If $f$ is a retract of $g$, and $g$ is a weak
equivalence/fibration/cofibration then so is $f$.

CM4 (lifting property): Let $i$ be a cofibration and $p$ be a
fibration. Then $p$ has the RLP with respect to $i$, if either $i$
or $p$ is a weak equivalence.

CM5 (decomposition property): any morphism $f$ has a decomposition

(a) $f=p\circ i$, where $p$ is fibration, and $i$ is a cofibration
and weak equivalence;

(b) $f=q\circ j$, where $q$ is fibration and weak equivalence, and
$j$ is a cofibration.
\end{Def}

\begin{Emp} \label{E:trivfib}
{\bf Notation.} (a) A map in a model category ${\C{C}}$ is called
a {\em trivial cofibration} (resp. {\em trivial fibration}) if it
is cofibration (resp. fibration) and a weak equivalence.

(b) By CM5, every morphism $f:X\to Y$ can be written as a
composition $X\overset{i}{\to}X'\overset{p}{\to}Y$, where $i$ is a
trivial cofibration, and $p$ a fibration. In such a case, we 
say that $p$ is a {\em fibrant replacement} of $f$.

(c) An element $X\in {\C{C}}$ is called {\em fibrant} (resp. {\em
cofibrant}), if the canonical map $X\to\pt$ (resp. $\emptyset\to
X$), where $\pt$ (resp $\emptyset$) is the final (resp. initial)
object of ${\C{C}}$, is a fibration (resp. cofibration).


\end{Emp}

For the following basic fact see, for example, \cite[II, Lem.
1.1]{GJ}.

\begin{Lem} \label{L:modcat}
A map $f:X\to Y$ in a model category ${\C{C}}$ is a cofibration
(resp. trivial cofibration) if and only if it has the LLP with
respect to all trivial fibrations (resp. fibrations).

(b) A map $f:X\to Y$ in a model category ${\C{C}}$ is a fibration
(resp. trivial fibration) if and only if it has the RLP with
respect to all trivial cofibrations (resp. cofibrations).
\end{Lem}

\begin{Emp} \label{E:remmodcat}
{\bf Remarks.} (a) \rl{modcat} implies in particular that
(trivial) cofibrations and (trivial) fibrations are closed under
compositions, and that all isomorphisms are trivial cofibrations and trivial 
fibrations.

(b) It also follows immediately from \rl{modcat} that (trivial)
fibrations are preserved by all pullbacks, and that (trivial)
cofibrations are preserved by all pushouts.

(c) It follows from CM5 (a) and CM2 that every weak equivalence
$f$ has a decomposition $f=p\circ i$, where $p$ is trivial
fibration, and $i$ is a trivial cofibration.
\end{Emp}


\begin{Def} \label{D:Cart}
We call a model category ${\C{C}}$ {\em Cartesian}, if ${\C{C}}$
is a Cartesian category, the final object of ${\C{C}}$ is
cofibrant, and for every cofibration $i:A\to B$ and fibration
$p:X\to Y$, the induced map $q:X^B\to X^A\times_{Y^A} Y^B$ is a
fibration and, additionally, $q$ is a weak equivalence if either
$i$ or $p$ is.
\end{Def}

\begin{Emp} \label{E:Cart}
{\bf Remark.} Taking $A=\emptyset$ or $Y=\pt$ in the definition of
Cartesian model category, we get the following particular cases.

(a) If $B$ is cofibrant, then for every (trivial) fibration $X\to
Y$, the induced map $X^B\to Y^B$ is a (trivial) fibration.

(b) If $X$ is fibrant, then for every (trivial) cofibration $A\to
B$, the induced map $X^B\to X^A$ is a (trivial) fibration.
\end{Emp}

\begin{Lem} \label{L:Cart}
Let ${\C{C}}$ be a model category, which is Cartesian as a
category, and such that the final object of ${\C{C}}$ is cofibrant.
Then ${\C{C}}$ is a Cartesian model category if and only if for
every two cofibrations $i:A\to B$ and $i':A'\to B'$, the
induced morphism $j:(A\times B')\sqcup_{(A\times A')}(B\times 
A')\to B\times B'$ is a cofibration and, additionally, $j$ is a
weak equivalence, if either $i$ or $i'$ is.
\end{Lem}

\begin{proof}
This follows from a combination of \rl{modcat} and \rl{RLP}.
\end{proof}

\begin{Def} \label{D:proper}
A model category ${\C{C}}$ is called:

(a) {\em right proper}, if weak equivalences are preserved by
pullbacks along fibrations;

(b) {\em left proper}, if weak equivalence are preserved by
pushouts along cofibrations;

(c) {\em proper}, if it is both left and right proper.
\end{Def}


%
%


\subsection{Simplicial sets}
\begin{Emp} \label{E:indcat}
{\bf Category $\Dt$.} (a) For $n\geq 0$, we denote by $[n]$ the
category, corresponding to a partially ordered set
$\{0<1<\ldots<n\}$. Let $\Dt$ be the full subcategory of the
category of small categories $Cat$, consisting of objects $[n],
n\geq 0$.

(b) For each $(m+1)$-tuple of integers $0\leq k_0\leq k_1\leq\ldots \leq k_m\leq n$, we denote by
$\dt^{k_0,\ldots,k_m}$ the map $\dt:[m]\to[n]$ such that
$\dt(i)=k_i$ for all $i$.

(c) For $0\leq i\leq n$ we define an inclusion $d^i:[n-1]\hra
[n]$ such that $i\notin\im d^i$; for $0\leq i<j\leq n$, we define an 
inclusion $d^{i,j}:[n-2]\hra [n]$ such that $i,j\notin\im d^{i,j}$;
for $0\leq i\leq n-m$ we define an inclusion $e^i:[m]\to[n]$
defined by $e^i(k):=k+i$.
\end{Emp}

\begin{Emp} \label{E:sp}
{\bf Spaces.} (a) By the {\em category of spaces} or, what is the same,
the {\em category simplicial sets} we mean the category of functors
$Sp:=\Fun(\Dt^{op},Set)$.

(b) For $X\in Sp$, we set $X_n:=X([n])$. For every
$\tau:[n]\to[m]$ in $\Dt$, we denote by  $\tau^*:X_m\to X_n$ the
induced map of sets. For every morphism $f:X\to Y$ in $Sp$ we
denote by $f_n$ the corresponding map $X_n\to Y_n$.

(c) By \re{expsh}, category $Sp$ is Cartesian and has all limits
and colimits.
\end{Emp}

\begin{Emp} \label{E:stsym}
{\bf The standard $n$-simplex.} (a) For every $n\geq 0$, we denote
by $\Dt[n]\in Sp$ the functor $\Hom_{\Dt}(\cdot,[n]):\Dt^{op}\to
Set$. Then $\pt:=\Dt[0]$ is a
final object of $Sp$.

(b) The Yoneda lemma  defines identifications $\Hom_{Sp}(\Dt[n],X)=X_n$ and \\
$\Hom_{Sp}(\Dt[n],\Dt[m])=\Hom_{\Dt}([n],[m])$ for all $X\in Sp$ and all $n,m\geq 0$.

(c) We denote by $\Dt^i[n]$ the image of the inclusion
$d^i:\Dt[n-1]\to\Dt[n]$, and set $\p\Dt[n]:=\cup_{i=0}^n\Dt^i[n]\subset\Dt[n]$ 
and $\La^k[n]:=\cup_{i\neq k}\Dt^k[n]\subset\Dt[n]$ for all
$k=0,\ldots,n$.
\end{Emp}

\begin{Emp} \label{E:fibers}
{\bf Fibers.} (a) For  $X\in Sp$, we say $x\in X$ instead of $x\in
X_0$. By \re{stsym} (b), each $x\in X$ corresponds to a map
$x:\pt\to X$.

(b) For every morphism $f:Y\to X$, we denote by $f^{-1}(x)$ or
$Y_x$ the fiber product $\{x\}\times_X Y:=\pt\times_{x,X}Y$ and
call it {\em the fiber of $f$ at $x$}.

(c) For every $Z\in Sp$ and $X,Y\in Sp/Z$, we denote by
$\Map_Z(X,Y)$ the fiber of $Y^X\to Z^X$ over the
projection $(X\to Z)\in Z^X$. 

\end{Emp}

\begin{Def} \label{D:kanms}
(a) A map $f:X\to Y$ in $Sp$ is called a {\em (Kan) fibration}, if
it has the RLP with respect to inclusions $\La^k[n]\hra\Dt[n]$ for
all $n>0,k=0,\ldots,n$.

(b) A map $f:X\to Y$ in $Sp$ is called a {\em weak equivalence},
if it induces a weak equivalence $|f|:|X|\to |Y|$ between
geometric realisations (see \cite[p. 60]{GJ}).

(c)  A map $f:X\to Y$ in $Sp$ is called a {\em cofibration}, if $f_n:X_n\to Y_n$ is an inclusion 
for all $n$. 
\end{Def}

\begin{Thm} \label{T:Kan}
Category $Sp$ has a structure of a proper Cartesian model category
such that cofibrations, fibrations
and weak equivalences are defined in \rd{kanms}. In particular,
all $X\in Sp$ are cofibrant, and trivial fibrations are precisely the maps
which have the RLP with respect to inclusions
$\p\Dt[n]\hra\Dt[n],n\geq 0$.
\end{Thm}

\begin{proof}
See \cite[I, Thm 11.3, Prop 11.5 and II, Cor 8.6]{GJ} and note
that in the case of model category $Sp$, "Cartesian"  means the
same as "simplicial".
\end{proof}

\begin{Def} \label{D:kancoml}
We say that $X\in Sp$ is a {\em (contractible) Kan complex}, if
the projection $X\to\pt$ is a (trivial) fibration.
\end{Def}

\begin{Emp} \label{E:pi0}
{\bf Connected components}. (a) We say that $X\in Sp$ is {\em
connected}, if it can not be written as $X=X'\sqcup X''$, where
$X',X''\neq\emptyset$. We say that $Y\subset X$ is a {\em
connected component of $X$}, if it is a maximal connected subspace
of $X$. Notice that $X$ is a disjoint union of its connected
components.

(b) We denote the set of connected components of $X$ by $\pi_0(X)$. Then every
map $f:X\to Y$ in $Sp$ induces a map
$\pi_0(f):\pi_0(X)\to\pi_0(Y)$. Note
that $X$ is connected if and only if its geometric realization
$|X|$ is connected. In particular, we have an equality
$\pi_0(X)=\pi_0(|X|)$. Therefore for every weak equivalence
$f:X\to Y$ in $Sp$, the map $\pi_0(f)$ is a bijection.

(c) For $x,y\in X$, we say that $x\sim y$, if $x$ and $y$ belong to the
same connected component of $X$. If $X$ is a Kan complex,
then $x\sim y$ if and only if there exists a map $\al:\Dt[1]\to X$
such that $\al(0)=x$ and $\al(1)=y$ (see \cite[Lem 6.1]{GJ}).
\end{Emp}

\begin{Lem} \label{L:propfib}
(a) Let $f:X\to Y$ be a trivial fibration. Then the space of
sections $\Map_Y(Y,X)$ of $f$ is non-empty and connected.

(b) Let $f:X\to Y$ be a fibration in $Sp$. Then $f$ is trivial if
and only if the Kan complex $f^{-1}(y)$ is contractible for every
$y\in Y$.

(c) Let $f:X\to Y$ is a map of fibrations over $Z$ in $Sp$. Then
$f$ is a weak equivalence if and only if the map of fibers
$f_z:X_z\to Y_z$ is a weak equivalence for every $z\in Z$.
\end{Lem}

\begin{proof}
(a) Since $Y$ is cofibrant, the projection $X^Y\to Y^Y$ is a
trivial fibration (by \re{Cart} (a)). Hence its fiber
$\Map_Y(Y,X)$ is a contractible Kan complex (by \re{remmodcat}
(b)), therefore it is non-empty and connected by \re{pi0} (b).

(b) By the last assertion of \rt{Kan}, the fibration $f$ is
trivial if and only if its pullback $\tau^*(f)$ is a trivial
fibration for all $\tau:\Dt[n]\to Y$. Thus we may assume that
$Y=\Dt[n]$. Then for each $y\in\Dt[n]$, the inclusion
$y:\Dt[0]\to\Dt[n]$ is a weak equivalence. Thus $X_y\to X$ is a
weak equivalence, because $Sp$ is right proper. Hence, by
2-out-of-3,  $f$ is a weak equivalence if and only if
$X_y\to\Dt[0]$ is.

(c) will  be proven in \re{pfpropfib}.
\end{proof}


\subsection{Simplicial Spaces}

\begin{Emp} \label{E:ssp}
{\bf Notation.} (a) By the category of {\em simplicial spaces}, we
mean the category of functors
$sSp=\Fun(\Dt^{op},Sp)=\Fun(\Dt^{op}\times\Dt^{op},Set)$.

(b) For $X\in sSp$ and $n,m\geq 0$, we set $X_n:=X([n])\in Sp$
and $X_{n,m}:=(X_n)_m\in Set$. For every morphism $f:X\to Y$ in
$sSp$, we denote by $f_n:X_n\to Y_n$ the corresponding morphism in
$Sp$.

(c) For every $\tau:[n]\to[m]$ in $\Dt$, we denote by
$\tau^*:X_m\to X_n$ the induced map of spaces. We also set
$\dt_{k_0,\ldots,k_m}:=(\dt^{k_0,\ldots,k_m})^*:X_n\to X_m$,
$d_i:=(d^i)^*:X_n\to X_{n-1}$, and $e_i:=(e^i)^*:X_n\to X_m$.

(d) By \re{expsh}, category $sSp$ is Cartesian and has all limits
and colimits. For $X,Y\in sSp$, we define the mapping space
$\Map(Y,X):=(X^Y)_0\in Sp$.
\end{Emp}

\begin{Emp}
{\bf Two embeddings $Sp\hra sSp$.} (a) Denote by $diag:Sp\to sSp$ (resp. $diag:Set\to Sp$)
the map which associates to each $X$ the constant simplicial
space (resp. set) $[n]\mapsto X,\tau\mapsto\Id_X$.
For each $X\in Sp$, we denote the constant simplicial space $diag(X)\in sSp$ simply by $X$.  

(b) The embedding $diag:Set\to Sp$ gives rise to an embedding
$disc:Sp=\Fun(\Dt^{op},Set)\to sSp=\Fun(\Dt^{op},Sp)$. Then the image of $disc$, 
consists of {\em discrete simplicial spaces}, that is, $X\in sSp$
such that $X_n\in Sp$ is {\em discrete} (that is, each map $\Dt[1]\to
X_n$ is constant) for all $n$.

(c) We set $F[n]:=disc(\Dt[n])$, $\p F[n]:=disc(\p\Dt[n])$ and
$F^i[n]:=disc(\La^i[n])$.
\end{Emp}

\begin{Emp} \label{E:stbis}
{\bf Standard bisimplex.} (a) For $n,m\geq 0$, we set
$[n,m]:=([n],[m])\in \Dt^2$ and $\Box[n,m]:=F[n]\times \Dt[m]\in
sSp$. In particular, we have equalities $F[n]=\Box[n,0]$, $\Dt[m]=\Box[0,m]$
and $\pt=F[0]=\Dt[0]$.  

(b) Note that $\Box[n,m]$ is the functor
$\Hom_{\Dt\times\Dt}(\cdot,[n,m])$. Then, by the Yoneda lemma, we get identifications 
$\Hom(\Box[n,m],\Box[n',m'])=\Hom([n,m],[n',m'])$ and 
$\Hom(\Box[n,m],X)=X_{n,m}$. In particular,
we have identifications \\ $\Map(F[n],X)=X_n$ and $\Hom(F[n],F[m])=\Hom([n],[m])$.

(c) We also set $\p\Box[n,m]:=(\p F[n]\times \Dt[m])\sqcup_{(\p
F[n]\times\p\Dt[m])}(F[n]\times\p\Dt[m])$ and \\ $X_{\p
n}:=\Map(\p F[n],X)$.
\end{Emp}

\begin{Emp} \label{E:fibers2}
{\bf Fibers.} (a) For  $X\in sSp$, we say $x\in X$ instead of $x\in
X_{0,0}$, and $x\sim y\in X$ instead of $x\sim y\in X_0$. By
\re{stbis} (b), each $x\in X$ corresponds to a map $x:\pt\to X$.

(b) As in \re{fibers}, for every morphism $f:Y\to X$ in $sSp$, we
denote by $f^{-1}(x)$ or $Y_x$ the fiber product $\{x\}\times_X
Y:=\pt\times_{x,X}Y$ and call it {\em the fiber of $f$ at $x$}.

(c) For every $Z\in sSp$ and $X,Y\in sSp/Z$ we denote by
$\C{Map}_Z(X,Y)\in sSp$ the fiber of $Y^X\to Z^X$ over the
projection $(X\to Z)\in Z^X$. We also set \\
$\Map_Z(X,Y):=\C{Map}_Z(X,Y)_0\in Sp$.

\end{Emp}

\begin{Def} \label{D:rfib}
We say that a map $f:X\to Y$ in $sSp$ is a {\em (Reedy) fibration}
if for every $n\geq 0$ the induced map $\bar{f}_n:X_n\to
Y_n\times_{Y_{\p n}}X_{\p n}$ is a Kan fibration in $Sp$.
\end{Def}

\begin{Thm} \label{T:Reedy}
Category $sSp$ has Cartesian proper model category such that
cofibrations and weak equivalences are degree-wise and fibrations
are Reedy fibrations.
\end{Thm}

\begin{proof}
All the assertions, except that the model category is Cartesian,
are proven in \cite[IV,Thm 3.9]{GJ}. For the remaining assertion, we 
use \rl{Cart}. Now the assertion follows from the fact
that pushouts, products, cofibrations and weak equivalences are
defined degree-wise, and the model category $Sp$ is Cartesian.
\end{proof}

\begin{Emp} \label{E:rfib}
{\bf Remarks.} (a) It follows from \rl{RLP} that a map $f:X\to Y$
in $sSp$ is a fibration if and only if it has the RLP with respect
to all inclusions\\
$(\p F[n]\times\Dt[m])\sqcup_{(\p F[n]\times\La^i[m])}
(F[n]\times\La^i[m])\hra\Box[n,m]$.

(b) If $f:X\to Y$ is a Reedy fibration, then the map $f_n:X_n\to Y_n$ is a fibration for all $n$. 
Indeed, $f_0=\bar{f}_0$ is a fibration by definition, 
$f^{F[n]}:X^{F[n]}\to Y^{F[n]}$ is a fibration by \re{Cart} (a), hence $f_n=(f^{F[n]})_0$ is a fibration.  

(c) Let $X\to Y$ be a fibration in $sSp$, and let $i:A\to B$ be a
cofibration over $Y$. Then the map $X^B\to
Y^B\times_{Y^A} X^A$ is fibration by \rt{Reedy}. Hence taking fibers 
at $(B\to Y)\in Y^B$ and passing
to zero spaces, we get that the map 
$i^*:\Map_Y(B,X)\to\Map_Y(A,X)$ is a fibration, thus 
 $\Map_Y(B,X)$ is a Kan complex. 
\end{Emp}

\begin{Emp} \label{E:he}
{\bf Homotopy equivalence.} Let $Z\in sSp$ (resp. $Z\in Sp$).

(a) We say that maps $f:X\to Y$ and $g:X\to Y$ in $sSp/Z$ (resp.
$Sp/Z$) are {\em homotopic over $Z$} and write $f\sim_Z g$, if
$f\sim g$ as elements of $\Map_Z(X,Y)$.

Notice that if $Y\to Z$ is a fibration, then  $\Map_Z(X,Y)\in Sp$ is
a Kan complex (by \re{rfib} (c)), thus by \re{pi0} (c) $f\sim_Z g$
means that there exists a map $h:X\times \Dt[1]\to Y$ over $Z$ such that $h|_0=f$ and
$h|_1=f$.

(b) We say that a map $f:X\to Y$ is a {\em homotopy equivalence}
over $Z$, if there exists a map $g:Y\to X$ over $Z$, called a
{\em homotopy inverse} of $f$, such that $f\circ g\sim_Z\Id_Y$ and
$g\circ f\sim_Z\Id_X$.
\end{Emp}

\begin{Emp} \label{E:herem}
{\bf Remarks.} (a) Let $f:X\to Y$ be a homotopy equivalence over
$Z$ with homotopy inverse $g$. Then for every $\tau:Z'\to Z$, the
pullback $\tau^*(f)$ is a homotopy equivalence over $Z'$ with
homotopy inverse $\tau^*(g)$. Similarly, for every $K\in sSp$, the
map $f^K:X^K\to Y^K$ is a homotopy equivalence with homotopy
inverse $g^K:Y^K\to X^K$. Also, a composition of homotopy
equivalences is a homotopy equivalence.

(b) Any homotopy equivalence is a weak equivalence. Indeed, the
assertion for $Sp$ follows from the fact that if $f\sim_Z g$, then
the geometric realizations satisfy $|f|\sim_{|Z|}|g|$, and the
assertion for $sSp$ follows from that for $Sp$.
\end{Emp}

\begin{Emp} \label{E:sdr}
{\bf Strong deformation retract.} Let $Z\in sSp$ (resp. $Z\in
Sp$).

(a) We say that an inclusion $i:Y\hra X$ over $Z$ is a {\em strong
deformation retract over $Z$}, if there exists a map $h:X\times
\Dt[1]\to X$ over $Z$ such that $h|_0=\Id_X$, and $h|_1(X)\subset Y$,
$h|_{Y\times\Dt[1]}$ is $Y\times\Dt[1]\overset{\pr_2}{\lra}Y\hra
X$.

(b) If $i:Y\hra X$ is a strong deformation retract over $Z$, then
$i$ is a homotopy equivalence over $Z$, and $h|_1:X\to Y$ is its
homotopy inverse. Also in this case, $Y\to Z$ is a retract of
$X\to Z$. In particular, if $X\to Z$ is a fibration, then $Y\to Z$
is a fibration as well (by CM3).

(c) Conversely, a trivial cofibration $i:Y\to X$ between
fibrations over $Z$ is a strong deformation retract over $Z$.

\begin{proof}
Since $i:Y\to X$ is trivial cofibration, while $Y\to Z$ is a
fibration, there exists a map $p:X\to Y$ over $Y$ such that
$p\circ i=\Id_{Y}$. Next, the induced map $(\p\Dt[1]\times
X)\sqcup_{(\p\Dt[1]\times Y)}(\Dt[1]\times Y)\hra\Dt[1]\times X$
is a trivial cofibration (see \rl{Cart}). Since $X\to Z$ is a
fibration, there exists a map $h:\Dt[1]\times X\to X$ over $Z$ such that
$h|_0=\Id_X$, $h|_1=p$ and $h|_{\Dt[1]\times Y}=\pr_2$.
\end{proof}
\end{Emp}

\begin{Lem} \label{L:heprop}
(a) A morphism $f:X\to Y$ over $Z$ is a homotopy equivalence over
$Z$ if and only if for every map $\al:K\to Z$ in $sSp$, the induced map
$\pi_0(\Map_Z(K,X))\to \pi_0(\Map_Z(K,Y))$ is a bijection.

(b) Every weak equivalence between fibrations is a homotopy
equivalence. In particular, a pullback of a weak equivalence
between fibrations is a weak equivalence.

(c) For each fibration $X\to Y\times\Dt[1]$, there exists a weak
equivalence $X|_0\to X|_1$ of fibrations over $Y$.

(d) Let $f:Y_A\to A$ be a fibration and let $A\hra B$ be trivial
cofibration. Then there exists a fibration $g:Y_B\to B$, whose
restriction to $A$ is $f$.
\end{Lem}

\begin{proof}
(a) If $f$ is a homotopy equivalence, then the induced map
$\Map_Z(K,X)\to\Map_Z(K,Y)$ is a homotopy equivalence (by
\re{herem} (a)), thus the assertion follows from \re{herem} (b)
and \re{pi0} (b). Conversely, applying the assumption for the
projection $Y\to Z$, we find a morphism $g:Y\to X$ over $Z$ such
that $f\circ g\sim_Z \Id_Y$. Next applying it to the projection
$X\to Z$, we find that  $g\circ f\sim_Z \Id_X$.

(b) By \re{remmodcat} (c) and \re{herem} (a), it is enough to
consider separately cases of a trivial cofibration and a trivial
fibration. When $f$ is a trivial cofibration, the assertion follows
from \re{sdr} (c) and (b). When $f$ is a trivial fibration, the
assertion follows from (a). Indeed, each map $\Map_Z(K,X)\to
\Map_Z(K,Y)$ is a trivial fibration, thus the map on $\pi_0$ is a
bijection by \re{pi0} (b). The last assertion follows from
\re{herem} (a).

(c) Since each map $\dt^i:\Dt[0]\hra\Dt[1]$ is a trivial
cofibration, the induced map $(\dt^i)^*:X^{\Dt[1]}\to X\times_{\Dt[1]\times
Y}(Y\times \Dt[1])^{\Dt[1]}$ is a trivial
fibration. Taking the pullback with respect to the inclusion
$Y\hra(Y\times \Dt[1])^{\Dt[1]}$, corresponding to $\Id_{Y\times
\Dt[1]}$, we get a trivial fibration
$\wt{X}:=X^{\Dt[1]}\times_{(\Dt[1]\times Y)^{\Dt[1]}}Y\to X|_i$
over $Y$. Thus both $X|_0$ and $X|_1$ are homotopy equivalent to
$\wt{X}$ over $Y$ (by (b)), hence they are homotopy equivalent.

(d) will be proven in \re{pfhepropd}.
\end{proof}

\begin{Emp} \label{E:contr}
{\bf Remark.} It follows from \rl{heprop} (b) and \re{pi0} (b)
that a Kan complex $X\in Sp$ is contractible if and only if the
projection $X\to\pt$ is a homotopy equivalence. Thus by definition
this happens if and only if $X$ is non-empty and $\Id_X$ is
homotopic to a constant map $X\to\{x\}\subset X$.
\end{Emp}

\begin{Emp} \label{E:pfpropfib}
\begin{proof}[Proof of \rl{propfib} (c)]
If $f$ is a weak equivalence, then each $f_z:X_z\to Y_z$ is a
weak equivalence by \rl{heprop} (b).

Conversely, write $f$ as $p\circ i$, where $i:X\to X'$ 
is a trivial cofibration, and $p:X'\to Y$ is fibration. By the "only if" assertion, 
each $i_z$ is a weak equivalence. Since each $f_z$ is a trivial fibration by assumption,
each $p_z$ is a weak equivalence by 2-out-of-3. Since $p_z:X'_z\to Y_z$  is
a fibration, it is a trivial fibration. Hence all fibers of each $p_z$ are contractible.
Thus all fibers of $p$ are contractible, hence $p$ is a trivial
fibration by \rl{propfib} (b). Therefore $f$ is a weak equivalence.
\end{proof}
\end{Emp}

\subsection{Segal spaces} We follow closely \cite{Re}.

\begin{Emp} \label{E:Segal}
{\bf Notation.} (a) We say that $X\in sSp$ is a {\em Segal space},
if $X$ is fibrant, and
\[
\varphi_n=:\dt_{01}\times_{\dt_1}\ldots\times_{\dt_{n-1}}\dt_{n-1,n}:X_n\to
X_1\times_{X_0}\ldots\times _{X_0}X_1
\]
is a weak equivalence for each $n\geq 2$.

(b) Notice that since Reedy model category is Cartesian, when $X$
is fibrant, the map $\varphi_n$ is a fibration. Thus a
fibrant object $X\in sSp$ is a Segal space if and only if each map
$\varphi_n$ is a trivial fibration.
\end{Emp}

\begin{Emp} \label{E:objmap}
{\bf "Objects" and "Mapping spaces".} Let $X$ be a Segal space.

(a) By a {\em space of objects} of $X$ we mean space $X_0$. We set
$\Ob X:=X_{0,0}$ and call it {\em the set of objects} of $X$. As
in \re{fibers2}, we say $x\in X$ instead of $x\in\Ob X$.

(b) For each $x,y\in X$, we denote by $\map(x,y)=\map_X(x,y)\in
Sp$ the fiber of $(\dt_0,\dt_1):X_1\to X_0\times X_0$ over
$(x,y)$. Notice that since $X$ is fibrant, the map $(\dt_0,\dt_1)$
is a fibration, thus each space $\map(x,y)$ is a Kan complex. For each
$x,y,z\in X$,  we denote by $\map(x,y,z)$  the fiber of
$(\dt_0,\dt_1,\dt_2):X_3\to(X_0)^3$ over $(x,y,z)$.

(c) For each $x\in X$, we set $\id_x:=\dt_{0,0}(x)\in\map_X(x,x)$.

(d) We call a map between Segal spaces  $f:X\to Y$ is {\em fully
faithful}, if for every $x,y\in X$ the induced map
$\map_X(x,y)\to\map_Y(f(x),f(y))$ is a weak equivalence. 
\end{Emp}

\begin{Emp} \label{E:ho}
{\bf The homotopy category.} (a) Let $X$ be a Segal space, and
$x,y,z\in X$. Then the trivial fibration $\varphi_2:X_2\to
X_1\times_{X_0}X_1$ induces a trivial fibration
$\map(x,y,z)\to\map(x,y)\times\map(y,z)$ (by \re{remmodcat} (b)),
which by \rl{propfib} (a) has a section $s$, unique up to
homotopy. Thus we have a well-defined map
\[
[s]:=\pi_0(s):\pi_0(\map(x,y))\times\pi_0(\map(y,z))\to\pi_0(\map(x,y,z)).
\]

(b) The map $\dt_{02}:X_2\to X_1$ induces a map
$\dt_{02}:\map(x,y,z)\to \map(x,z)$. Therefore for every
$[\al]\in\pi_0(\map(x,y))$ and $[\beta]\in\pi_0(\map(y,z))$ we can
define 
\begin{equation} \label{Eq:comp}
[\beta]\circ[\al]:=\pi_0(\dt_{02})(
[s]([\al],[\beta]))\in\pi_0(\map(x,z)).
\end{equation}
It is not difficult to prove (see \cite[Prop. 5.4]{Re}) that this
composition is associative and satisfies 
$[\al]\circ[\id_x]=[\al]=[\id_y]\circ [\al]$ for all
$\al\in\map(x,y)$.

(c) Using (b), one can associate to $X$ its {\em homotopy
category} $\Ho X$, whose objects are $\Ob X$, morphisms defined by
$\Hom_{\Ho X}(x,y):=\pi_0(\map_X(x,y))$, the composition is
defined by \form{comp}, and the identity map is
$[\id_x]\in\Hom_{\Ho X}(x,x)$.
\end{Emp}

\begin{Emp} \label{E:css}
{\bf Complete Segal spaces.} Let $X$ be a Segal space.

(a) We say that $\al\in\map_X(x,y)\subset X_1$ is a {\em homotopy
equivalence}, if the corresponding morphism $[\al]\in \Mor \Ho X$
is an isomorphism. Explicitly, this means that there exist
$\beta\in \map(y,x)$ such that $[\beta]\circ[\al]=[\id_x]$ and
$[\al]\circ [\beta]=[\id_y]$.

(b) Let $X_{heq}\subset X_1$ be the maximal subspace such that each
$\al\in X_{heq}$ is a homotopy equivalence. It is not difficult to
prove (see \cite[Lem. 5.8]{Re}) that $X_{heq}\subset X_1$ is a
union of connected components.

(c) Notice that since each $[\id_x]$ is an isomorphism, we have
$\id_x\in X_{heq}$ for every $x\in X$. Therefore the map
$s_0:=\dt_{0,0}:X_0\to X_1$ factors through $X_{heq}$. We say that $X$ is
called a {\em  complete Segal space}, if the map $s_0:X_0\to
X_{heq}$ is a weak equivalence.
\end{Emp}

\begin{Lem} \label{L:css}
Let $X\in Sp$ be a Segal space.

(a) Let $X'_1\subset X_1$ be the union of connected components,
intersecting $s_0(X_0)$, and set
$X'_3:=\dt^{-1}_{02}(X'_1)\cap\dt_{13}^{-1}(X'_1)\subset X_3$.
Then $X'_1\subset X_{heq}$ and $\dt_{12}(X'_3)=X_{heq}$.

(b) $X$ is complete if and only if $\dt_0:X_{heq}\hra X_1\to X_0$
is a trivial fibration.
\end{Lem}

\begin{proof}
(a) Since $X_{heq}\subset X_1$ is a union of connected components,  inclusion
$s_0(X_0)\subset X_{heq}$ implies that $X'_1\subset X_{heq}$. 
Next, for each $\al\in X'_3$ we have
$\dt_{02}(\al),\dt_{13}(\al)\in X'_1\subset X_{heq}$, hence 
$[\dt_{02}(\al)]=[\dt_{12}(\al)]\circ
[\dt_{01}(\al)]$ and $[\dt_{13}(\al)]=[\dt_{23}(\al)]\circ
[\dt_{12}(\al)]$ are isomorphisms in $\Ho X$. Therefore 
$[\dt_{12}(\al)]$ is isomorphism, thus $\dt_{12}(\al)\in X_{heq}$.

Conversely, let $\al\in \map(x,y)\subset X_{heq}$ and let
$\beta\in \map(y,x)$ such that $[\beta]\circ[\al]=[\id_x]$ and
$[\al]\circ [\beta]=[\id_y]$. Since  $\varphi_3$ is a trivial
fibration, it is surjective. Thus there exists $\gm\in X_3$ such that
$\dt_{01}(\gm)=\dt_{23}(\gm)=\beta$ and $\dt_{12}(\gm)=\al$. Then,
by assumption, $\dt_{02}(\gm)\sim\id_y$ and
$\dt_{13}(\gm)\sim\id_x$, thus $\gm\in X'_3$.

(c) By \re{css} (b), the composition $X_{heq}\hra X_1\to X_0\times X_0$ is a
fibration, thus a projection $\dt_0:X_{heq}\to X_0$ is a
fibration. Since $\dt_0\circ s_0=\Id_{X_0}$, we conclude that
$s_0:X_0\to X_{heq}$ is a weak equivalence if and only if
$\dt_0:X_{heq}\to X_0$ is a trivial fibration.
\end{proof}

\begin{Emp} \label{E:CartSS}
{\bf Cartesian structure.} Rezk showed (see \cite[Cor 7.3]{Re})
that if $X$ is a (complete) Segal space, then $X^K$ is a
(complete) Segal space for every $K\in sSp$.
\end{Emp}

\section{The Yoneda lemma}

\subsection{Left fibrations}
\begin{Def} \label{D:left}
We call a fibration $f:X\to Y$ in $sSp$ a {\em left
fibration}, if the map $(f_*,(\dt^0)^*):X^{F[1]}\to X\times_Y
Y^{F[1]}$, induced by $\dt^0:F[0]\hra F[1]$, is a
trivial fibration.
\end{Def}

\begin{Lem} \label{L:left}
(a) A pullback of a left fibration is a left fibration.

(b) If $f:X\to Y$ is a left fibration, then $f^Z:X^Z\to Y^Z$ is a
left fibration for every $Z\in sSp$.
\end{Lem}

\begin{proof}
(a) follows from the fact that a pullback of a (trivial) fibration
is a (trivial) fibration (see \re{remmodcat} (b)).

(b) By definition, the map $X^{F[1]}\to X\times_Y Y^{F[1]}$ is a
trivial fibration. Since Reedy model structure is Cartesian, we
conclude that $f^Z$ is a fibration, while the map
$(X^Z)^{F[1]}=(X^{F[1]})^Z\to (X\times_Y
Y^{F[1]})^Z=X^Z\times_{Y_Z} (Y^Z)^{F[1]}$ is a trivial fibration (use \re{Cart}).
Thus $f^Z:X^Z\to Y^Z$ is a left fibration.
\end{proof}



\begin{Lem} \label{L:leftcart}
Let $f:X\to Y$ be a fibration in $sSp$. The following conditions
are equivalent:

(a) $f$ is a left fibration.

(b) For every $n\geq 1$, the map $(f_*,(\dt^0)^*):X^{F[n]}\to
X\times_Y Y^{F[n]}$, induced by $\dt^0:[0]\hra[n]$,
is a trivial fibration.

(c) For every $n\geq 1$, the map $p_n:X_{n}\to
X_0\times_{Y_0}Y_{n}$, induced by 
$\dt^0:[0]\hra[n]$, is a trivial fibration.
\end{Lem}

\begin{proof}
(a)$\implies(b)$ By (a) and \re{Cart} (a), the map
\[
p:X^{F[1]\times F[n]}=(X^{F[1]})^{F[n]}\to (X\times_Y
Y^{F[1]})^{F[n]}=X^{F[n]}\times_{Y^{F[n]}}Y^{F[1]\times F[n]}
\]
is a trivial fibration. Since trivial fibrations are stable under
retracts (axiom CM3), it remains to show that the map
$X^{F[n+1]}\to X\times_Y Y^{F[n+1]}$ is a retract of $p$. It is
enough to show that  $\dt^0:[0]\hra[n+1]$ is a retract of
$\dt^0:[n]\times [0]\hra [n]\times [1]$. Consider maps
$[n+1]\overset{\al}{\lra} [n]\times[1]\overset{\beta}{\lra}
[n+1]$, where $\al(0)=(0,0)$, $\al(i)=(i-1,1)$ for $i\geq 1$ and 
$\beta(i,j)=(i+1)j$. Then $\beta\circ\al=\Id$, $\al(0)\in
[n]\times\{0\}$ and $\beta([n]\times\{0\})=0$, thus $\al$ and
$\beta$ realize $\dt^0:[0]\hra[n+1]$ as a retract of $\dt^0:[n]\times
[0]\hra[n]\times [1]$.

(b)$\implies$(c) Pass to the zero spaces.

(c)$\implies$(a) First we assume that $Y$ is fibrant. Since $f$ is
a fibration, $X$ is fibrant, and the induced map $X^{F[1]}\to
X\times_Y Y^{F[1]}$ is a fibration. It remains to show that each
map $(X^{F[1]})_n\to X_n\times_{Y_n} (Y^{F[1]})_n$ is a weak
equivalence. Since the map $X_n\to X_0\times_{Y_0} Y_n$ is a
trivial fibration, its pullback 
\[
X_n\times_{Y_n} (Y^{F[1]})_n\to (X_0\times_{Y_0} Y_n)\times_{Y_n}
(Y^{F[1]})_n=X_0\times_{Y_0}(Y^{F[1]})_n
\]
is a trivial fibration. It remains to show that the map
$(X^{F[1]})_n\to X_0\times_{Y_0}(Y^{F[1]})_n$ or, equivalently,
$q_n:(X^{F[1]\times F[n]})_0\to X_0\times_{Y_0}(Y^{F[1]\times
F[n]})_0$ is a weak equivalence.

We follow the argument of \cite[Lem 10.3]{Re}. Let $\gm^i:[n+1]\to
[n]\times[1]$ (resp. $\epsilon^i:[n]\to [n]\times[1]$) be the map
with sends $j$ to $(j,0)$, if $j\leq i$ and to $(j-1,1)$ (resp.
$(j,1)$) otherwise. Then maps $\gm^i$ and $\epsilon^i$ induce 
decomposition of  $F[n]\times F[1]$ as \\
$F[n+1]\sqcup_{F[n]}\ldots\sqcup_{F[n]} F[n+1]$, where all maps
$F[n]\to F[n+1]$ are cofibrations.

Therefore we get decompositions of $(X^{F[n]\times F[1]})_0$ and
$X_0\times_{Y_0}(Y^{F[1]\times F[n]})_0$ as
\[(X^{F[n+1]})_0\times_{(X^{F[n]})_0}\ldots\times_{(X^{F[n]})_0}(X^{F[n+1]})_0\text{ and }\]
\[
(X_0\times_{Y_0}(Y^{F[n+1]})_0)\times_{(X_0\times_{Y_0}(Y^{F[n]})_0)}
\ldots\times_{(X_0\times_{Y_0}(Y^{F[n]})_0)}
(X_0\times_{Y_0}(Y^{F[n+1]})_0).
\]
Since $X$ and $Y$ are fibrant, all maps in
both fiber products are fibrations.

Thus $q_n$ can be written as a fiber products of $X_{n+1}\to
X_0\times_{Y_0}Y_{n+1}$'s over $X_n\to X_0\times_{Y_0}Y_n$'s. Since
these maps are weak equivalences by (c), we conclude that $p$ is a
weak equivalence by \rco{hfibprod}.

For a general $Y$, we choose a fibrant replacement $Y\hra Y'$. Then
by \rl{heprop} (d) there exists a fibration $f':X'\to Y'$, whose
restriction to $Y$ is $f$. We claim that $f'$ satisfies assumption
(c). Since $f'$ is a fibration, each  $p'_n:X'_{n}\to
X'_0\times_{Y'_0}Y'_{n}$ is a fibration. Thus it remains to show
that $p'_n$ is a weak equivalence.

Consider Cartesian diagram
\[
\begin{CD}
X_{n} @>p_n>> X_0\times_{Y_0}Y_n @>>> Y_n\\
@Vi''VV @Vi'VV   @ViVV   \\
X'_{n} @>p'_n>> X'_0\times_{Y'_0}Y'_n @>>> Y'_n
\end{CD}
\]
and note that all horizontal maps are fibrations. Since $i$ is a
weak equivalence and $Sp$ is right proper, we conclude that $i'$
and $i''$ are weak equivalences. Since $p_n$ is a weak equivalence
by assumption, $p'_n$ is a weak equivalence by 2-out-of-3.

By the application (c)$\implies$(a) for fibrant $Y$, the map $q':X'^{F[1]}\to
X'\times_{Y'} Y'^{F[1]}$ is a trivial fibration. Hence $q$, being
the restriction of $q'$ to $X\times_{Y} Y^{F[1]}$, is a trivial
fibration as well.
\end{proof}

\begin{Emp} \label{E:left}
{\bf Remarks.} (a) By \rl{leftcart} (c), a morphism  $f:X\to Y$ is
a left fibration if and only if it satisfies the RLP with respect to cofibrations 
\[
(F[n]\times\La^i[m])\sqcup_{(\partial
F[n]\times\La^i[m])}(\partial F[n]\times\Dt[m])\hra\Box[n,m]\]
\[(F[n]\times\partial\Dt[m])\sqcup_{(F[0]\times\partial\Dt[m])}
(F[0]\times\Dt[m])\hra\Box[n,m].\] In particular, a
morphism $f:X\to Y$ is a left fibration if and only if for every
morphism $\tau:\Box[n,m]\to Y$, the pullback
$\tau^*(f):\tau^*(X)\to\Box[n,m]$ is a left fibration.

(b) It also can be deduced from \rl{leftcart} (c) that if $f:X\to
Y$ is a left fibration and $Y$ is a (complete) Segal space, then
$X$ is a (complete) Segal space as well. We will not use this
fact.
\end{Emp}



\begin{Lem} \label{L:leftequiv}
A morphism $f:X\to Y$ of left fibrations over $Z$ is a weak
equivalence if and only if the map of fibers $f_z:(X_z)_0\to
(Y_z)_0$ is a weak equivalence  for each $z\in Z$.
\end{Lem}

\begin{proof}
Notice that $f_0:X_0\to Y_0$ is a morphism between fibrations over
$f_0$. Thus $f_0$ is a weak equivalence if and only if the induced
map $f_z:(X_z)_0\to (Y_z)_0$ between fibers is a weak equivalence
for all $z\in Z_0$ (by \rl{propfib} (c)). Thus it remains to show
that if $f_0$ is a weak equivalence, then $f_n$ is a weak
equivalence for each $n$. We have a commutative diagram
\[
\begin{CD}
X_n @>f_n>> Y_n\\
@VVV @VVV\\
X_0\times_{Z_0}Z_n  @>\wt{f_0}>> Y_0\times_{Z_0}Z_n,
\end{CD}
\]
whose vertical maps are trivial fibrations by \rl{leftcart} (c).
Since $f_0$ is a weak equivalence, while $X_0\to Z_0$ and $Y_0\to
Z_0$ are fibrations, the map $\wt{f_0}$ is a weak equivalence by
\rco{hfibprod}. Hence $f_n$ is a weak equivalence by  2-out-of-3.
\end{proof}

\begin{Emp} \label{E:section}
{\bf Undercategory.} (a) For $X\in sSp$ and $x\in X$, we set $x\bs
X:=\{x\}\times_X X^{F[1]}\to X$, where $X\to X^{F[1]}$ is induced by $s^0:F[1]\to F[0]$, 
and put $\id_x:=s_0(x)\in\{x\}\times_{X_0}X_1=(x\bs X)_0$.

(b) We claim that the projection $\pr_2:\id_x\bs(x\bs X)\to x\bs
X$ has a section $r$ such that $r(\id_x)=\id_{\id_x}$. Indeed, set
$A:=(F[1]\times\{0\})\cup(\{0\}\times F[1])\subset F[1]\times
F[1]$. Then $\id_x\bs (x\bs X)\subset X^{F[1]\times F[1]}$ can be
written as $\{x\}\times_{X^A} X^{F[1]\times F[1]}$. Therefore the
map $m:[1]\times[1]\to[1]$  defined by $m(i,j):=ij$ induces a
map $m:F[1]\times F[1]\to F[1]$ such that $m(A)=0$. Hence $m$
induces a map $r:x\bs X\to \id_x\bs (x\bs X)$, which 
satisfies $r(\id_x)=\id_{\id_x}$ and $\pr_2\circ r=\Id$.
\end{Emp}


The following result is one of the main steps in the proof of the 
Yoneda lemma.

\begin{Prop} \label{P:triv}
For every left fibration $\pi:E\to X$ and $x\in X$, the evaluation
map $\ev_{\id_x}:\Map_X(x\bs X, E)\to \Map_X(\{\id_x\},
E)=(E_x)_0$, induced by the inclusion $\{\id_x\}\hra x\bs X$, is a
trivial fibration.
\end{Prop}

\begin{proof}
Since $\{\id_x\}\hra x\bs X$ is a cofibration, while $E\to X$ is a
fibration, the map $\ev_{\id_x}$ is a fibration (see \re{rfib}
(c)). Therefore it remains to show that for each $\al\in E_x$, the
Kan complex $\Map_X(x\bs X, E)_{\al}:=\ev_{\id_x}^{-1}(\al)$ is
contractible (by \rl{propfib} (b)). Using remark \re{contr}, it
suffices to  show that the identity map of $\Map_X(x\bs X,
E)_{\al}$ factors through a contractible Kan complex.

 Since $E\to X$ is a left fibration, the
projection $E^{F[1]}\to E\times_X X^{F[1]}$ is a trivial
fibration. Thus $\al\bs E\to x\bs X$, being its fiber over $\al\in
E$, is a trivial fibration. Therefore the evaluation map
$\ev'_{\id_x}:\Map_X(x\bs X, \al\bs E)\to ((\al\bs E)_{\id_x})_0$
is a fibration between contractible Kan complexes. Hence
$\ev'_{\id_x}$ is a weak equivalence, thus a trivial fibration.
Therefore its fiber $\ev'^{-1}_{\id_x}(\id_{\al})=\Map_{x\bs
X}(x\bs X,\al\bs E)_{\id_{\al}}$ is a contractible Kan complex.

Note that the projection $\pr_2:\al\bs E\to E$ induces a projection
\[
\rho:\Map_{x\bs X}(x\bs X,\al\bs E)_{\id_{\al}}\to \Map_X(x\bs X,
E)_{\al}.
\]
Thus it remains to show that $\rho$ has a section. 

The natural morphism $\Map(x\bs X, E)\to \Map((x\bs X)^{F[1]},
E^{F[1]})$ induces a morphism $s':\Map_X(x\bs X, E)_{\al}\to
\Map_{x\bs X}(\id_x\bs (x\bs X), \al\bs E)$. By \re{section}, the
projection $\pr_2:\id_x\bs(x\bs X)\to x\bs X$ has a section $r$
such that $r(\id_x)=\id_{\id_x}$. Then
\[
r^*\circ s': \Map_X(x\bs X, E)_{\al}\to \Map_{x\bs X}(\id_x\bs
(x\bs X), \al\bs E)\to\Map_{x\bs X}(x\bs X, \al\bs E)
\]
has an image in $\Map_{x\bs X}(x\bs X,\al\bs E)_{\id_{\al}}$ and
is a section of $\rho$.
\end{proof}

\begin{Cor} \label{C:trfib}
Let $\pi:E\to X$ be a left fibration, $x\in X$,  and let $f$ and
$g$ be maps $x\bs X\to E$ over $X$ such that $f(\id_x)\sim
g(\id_x)\in E_x$. Then $f\sim_X g$. In particular, $f$ is a weak
equivalence if and only if $g$ is a weak equivalence.
\end{Cor}

\begin{proof}
Since $\ev_{\id_x}$ is a trivial fibration (by \rp{triv}), the
induced map $\pi_0(\ev_{\id_x})$ is a bijection by \re{pi0} (b).
\end{proof}

\begin{Emp} \label{E:remss}
{\bf Remarks.} Let $X$ be a Segal space. (a) Then
$(\dt_{01},\dt_{12}):X_2\to X_1\times_{X_0}X_1$ and its pullback
$\dt_{12}:X_0\times_{s_0,X_1,\dt_{01}}X_2\to X_1$ are trivial
fibrations.

(b) The map $\dt_{02}:X_0\times_{s_0,X_1,\dt_{01}}X_2\to X_1$ is a
fibration. Indeed, $\dt_{12}$ is a pullback of the map
$(\dt_{01},\dt_{02}):X_2\to X_1\times_{X_0}X_1$, induced by the
inclusion $\dt^{01}F[1]\cup\dt^{02}F[1]\hra F[2]$. Thus it is a
fibration, because $X$ is fibrant.

(c) The map $\dt_{02}$ from (b) is a weak equivalence. Indeed, the
map $r=(\dt_0,\dt_{001}):X_1\to X_0\times_{s_0,X_1,\dt_{01}}X_2$
satisfy $\dt_{12}\circ r=\dt_{02}\circ r=\Id$. Since $\dt_{12}$ is
a weak equivalence (by (a)), we deduce that $r$ and $\dt_{02}$ are
weak equivalences by 2-out-of-3.
\end{Emp}

\begin{Lem} \label{L:undcat}
Let $X$ be a Segal space and $x\in X$. Then $x\bs X\to X$ is a
left fibration.
\end{Lem}

\begin{proof}
Since $X$ is fibrant, the projection $X^{F[1]}\to X^{\p
F[1]}=X\times X$ is fibration, hence its pullback $x\bs X\to X$ is
a fibration. It remains to show that the map $(x\bs X)^{F[1]}\to
(x\bs X)\times_X X^{F[1]}$ is a weak equivalence, or,
equivalently, that the map
 $((x\bs X)^{F[1]})_n\to
(x\bs X)_n\times_{X_n} (X^{F[1]})_n$ is a weak equivalence for all
$n$.

Using identifications $(X^{F[m]})_n=\Map(F[m]\times
F[n],X)=(X^{F[n]})_m$, we can rewrite the last map in the form
$(x\bs X^{F[n]})_1\to (x\bs X^{F[n]})_0\times_{(X^{F[n]})_0}
(X^{F[n]})_1$. Since $X^{F[n]}$ is also a Segal space (see
\re{CartSS}), we can replace $X$ by $X^{F[n]}$. It remains to show
that the map $(x\bs X)_1\to(x\bs X)_0\times_{X_0} X_1$ is a
trivial fibration.

Using decomposition $F[1]\times F[1]=F[2]\sqcup_{F[1]} F[2]$, we
get a decomposition
$(X^{F[1]})_1=X_2\times_{\dt_{02},X_1,\dt_{02}}X_2$. Hence we get
a decomposition
\begin{equation} \label{Eq:decomp}
(x\bs X)_1=(\{x\}\times_{X_1,\dt_{01}}X_2)\times_{(x\bs
X)_0}(\{x\}\times_{X_0,\dt_0}X_2),
\end{equation}
which identifies the map $(x\bs X)_1\to(x\bs X)_0\times_{X_0} X_1$
with a composition
\[
(x\bs
X)_1\overset{f}{\lra}\{x\}\times_{X_0,\dt_0}X_2\overset{g}{\lra}(x\bs
X)_0\times_{X_0} X_1.
\]
We claim that $f$ and $g$ are trivial fibrations. Since
$g$ is a pullback of $(\dt_{01},\dt_{12}):X_2\to
X_1\times_{X_0}X_1$, while $f$ is a pullback of
$\dt_{02}:\{x\}\times_{X_1,\dt_{01}}X_2\to(x\bs X)_0$, hence a
pullback of $\dt_{02}:X_0\times_{s_0,X_1,\dt_{01}}X_2\to X_1$,
both assertions follow from \re{remss}.
\end{proof}

\subsection{The $\infty$-category of spaces}

\begin{Emp} \label{E:overcat}
{\bf Overcategories.} (a) For each $K\in sSp$ we denote by $[K]$
the category of "bisimplexes of $K$". Explicitly, the set objects
of $[K]$ is the disjoint union $\sqcup_{n,m} K_{n,m}$ and for
every $a\in K_{n,m}$ and $b\in K_{n',m'}$ the set of morphisms
$\Mor_{[K]}(a,b)$ is the set of 
$\tau\in\Mor_{\Dt\times\Dt}([n',m'],[n,m])$ such that
$\tau^*(a)=b$.

(b) Note that we have a natural isomorphism of categories
$sSp/K\to \Fun([K],Set)$. Namely, each map $f:X\to K$ defines 
a functor $[K]\to Set$, which sends $a\in K_{n,m}$ to $f^{-1}_{n,m}(a)\subset
X_{n,m}$. Conversely, every $\phi:[K]\to Set$ gives rise to
$X_{\phi}\in sSp/K$, where $(X_{\phi})_{n,m}:=\sqcup_{a\in
K_{n,m}} \phi(a)$ with obvious transition maps.

(c) Every map $\phi:L\to K$ in $sSp$ induces a functor
$[\phi]:[L]\to[K]$. Then the bijection of (b) identifies
$\phi^*:sSp/K\to sSp/L$ with the pullback functor
 $[\phi]^*:\Fun([K],Set)\to\Fun([L],Set)$.
\end{Emp}

\begin{Emp} \label{E:universes}
{\bf Universes.} From now on we fix an infinite set $\C{U}$, which
we call a {\em universe}.

(a) Let $Set_{\C{U}}\subset Set$ be the category of subsets of
$\C{U}$, and let $Set_{|\C{U}|}$ the category of sets of
cardinality $\leq |\C{U}|$. Then category $Set_{\C{U}}$ is small,
and the natural embedding $Set_{\C{U}}\to Set_{|\C{U}|}$ is an
equivalences of categories.

(b) We set $Sp_{\C{U}}:=\Fun(\Dt^{op},Set_{\C{U}})\subset Sp$ and
$sSp_{\C{U}}:=\Fun(\Dt^{op},Sp_{\C{U}})\subset sSp$.

(c) More generally, for every $K\in sSp$, we denote by
$(sSp/K)_{\C{U}}\subset sSp/K$ (resp. $(sSp/K)_{|\C{U}|}\subset
sSp/K$) the full subcategory of morphisms $f:X\to K$ such that
fibers of all $f_{n,m}:X_{n,m}\to K_{n,m}$ belong to $Set_{\C{U}}$
(resp. $Set_{|\C{U}|}$).

(d) Bijection of \re{overcat} (b) induces a bijection between
$(sSp/K)_{\C{U}}$ (resp.  $(sSp/K)_{|\C{U}|}$) and functors
$[K]\to Set_{\C{U}}$ (resp. $[K]\to Set_{|\C{U}|}$). In
particular, category $(sSp/K)_{\C{U}}$ is small, and the inclusion 
$(sSp/K)_{\C{U}}\to (sSp/K)_{|\C{U}|}$ is an equivalence
of categories.

(e) We denote by $(LFib/K)_{\C{U}}$ the set of left fibrations $X\to
K$, belonging to $(sSp/K)_{\C{U}}$. By (d), \re{overcat}
(c) and \rl{left} (a), for every map $\phi:L\to K$, the pullback
functor $\phi^*:sSp/K\to sSp/L$ maps $(LFib/K)_{\C{U}}$ to
$(LFib/L)_{\C{U}}$.
\end{Emp}




\begin{Emp} \label{E:leftfibr}
{\bf Main construction.} (a) Let $\S_{\C{U}}\in sSp$ be the
simplicial space such that

\noindent $\bullet$ $(\S_{\C{U}})_{n,m}$ is the set of left
fibrations $(LFib/\Box[n,m])_{\C{U}}$;

\noindent $\bullet$ for every $a\in (\S_{\C{U}})_{n,m}$ with the
corresponding left fibration $E_{a}\to\Box[n,m]$ and every
$\nu:[n',m']\to[n,m]$, we have $E_{\nu^*(a)}=\nu^*(E_a)$ (use
remark \re{universes} (e)).

(b) Consider the "universal left fibration"
$p_{\C{U}}:\C{E}_{\C{U}}\to \S_{\C{U}}$, where
$(\C{E}_{\C{U}})_{n,m}$ is defined to be the disjoint union
$\sqcup_{a\in (\S_{\C{U}})_{n,m}} (E_{a})_{n,m}$, and $p_{\C{U}}$ 
is the map, which maps each $(E_a)_{n,m}$ to $a\in
(\S_{\C{U}})_{n,m}$.
\end{Emp}

\begin{Lem} \label{L:univ}
The map $p_{\C{U}}:\C{E}_{\C{U}}\to \S_{\C{U}}$ is a left
fibration. For each $K\in sSp$, the map $\phi\mapsto
\phi^*(p_{\C{U}})$ defines a bijection between
$\Hom_{sSp}(K,\S_{\C{U}})$ and $(LFib/K)_{\C{U}}$.
\end{Lem}

\begin{proof}
By construction, for every
$a\in(\S_{\C{U}})_{n,m}=\Hom(\Box[n,m],\S_{\C{U}})$, the pullback
$a^*(p_{\C{U}})$ equals $E_{a}\to\Box[n,m]$. In particular, each
$a^*(p_{\C{U}})$ is a left fibration. Thus $p$ is a left fibration
by remark \re{left} (a).

Next notice that for every $\phi:K\to \S_{\C{U}}$, the pullback
$\phi^*(p_{\C{U}}):\phi^*(\C{E}_{\C{U}})\to K$ is a left
fibration, satisfying $a^*(\phi^*(p_{\C{U}}))=(\phi\circ
a)^*(p_{\C{U}})\in (LFib/\Box[n,m])_{\C{U}}$ for each
$a:\Box[n,m]\to K$. Thus $\phi^*(p_{\C{U}})\in (LFib/K)_{\C{U}}$.

Conversely, every  $E\in (LFib/K)_{\C{U}}$ defines a map
$\phi_E:K\to \S_{\C{U}}$, which sends $a\in
K_{n,m}=\Hom(\Box[n,m],K)$ to the left fibration
$a^*(E)\to\Box[n,m]$ in $(LFib/\Box[n,m])_{\C{U}}$. Then the map
$E\mapsto \phi_E$ is inverse to $\phi\mapsto \phi^*(p_{\C{U}})$.
\end{proof}


\begin{Emp} \label{E:remsu}
{\bf Remarks.} (a) The main result of this subsection
(\rt{spaces}) asserts that $\S_{\C{U}}$ is a complete Segal
space. It is our model for the $\infty$-category of spaces, or,
more formally, the $(\infty,1)$-category of
$(\infty,0)$-categories.

(b) It can be shown that every inclusion $i:\C{U}\hra\C{V}$ of
infinite sets induces a fully faithful map $i:\S_{\C{U}}\hra
\S_{\C{V}}$ of complete Segal spaces.

(c) One can show (see \cite{KV2}) that $\S_{\C{U}}$ is equivalent
to the fibrant replacement $N^f(Sp_{\C{U}},W)$ of the simplicial
space $N(Sp_{\C{U}},W)$, associated by Rezk (\cite[3.3]{Re}) to
the pair $(Sp_{\C{U}},W)$, where $W$ denotes weak equivalences.

(d) One can also consider the "large" $(\infty,1)$-category of
$(\infty,0)$-categories $\wh{\S}$ such that $\wh{\S}_{n,m}$ is the
class of all left fibrations $E\to\Box[n,m]$.

(e) In \cite{KV2} we generalize \re{leftfibr} and construct the
$(\infty,n+1)$-category of $(\infty,n)$-categories.
\end{Emp}

\begin{Emp}
{\bf Notation.} (a) For every $n\geq 1$ denote by
$\S^{(n)}_{\C{U}}\in sSp$ the simplicial space such that
$(\S^{(n)})_{m,k}$ is the set of diagrams
$\phi:E^{(0)}\overset{\phi_1}{\lra} \ldots
\overset{\phi_n}{\lra}E^{(n)}$ over $\Box[m,k]$, where each
$E^{(i)}\to\Box[m,k]$ belongs to $(LFib/\Box[m,k])_{\C{U}}$.

(b) To every map $\mu:[m]\to[n]$ we associate morphism
$\mu^*:\S^{(n)}_{\C{U}}\to\S^{(m)}_{\C{U}}$, which sends diagram
$\phi:E^{(0)}\overset{\phi_1}{\lra} \ldots
\overset{\phi_n}{\lra}E^{(n)}$ to a diagram
$\mu^*(\phi):E^{(\mu(0))}\to \ldots\to E^{(\mu(m))}$, whose
morphisms are compositions of the $\phi_i$'s.

(c) Let $\S_{\C{U}}^{we}\subset \S^{(1)}_{\C{U}}$ be a simplicial
subspace such that $(\S^{(we)})_{m,k}\subset (\S^{(1)})_{m,k}$ consists of diagrams 
consists of diagrams $E^{(0)}\overset{\phi}{\lra} E^{(1)}$, where $\phi$ is a weak equivalence 
(use \rl{heprop} (b)).

(d) We have a natural projection $\S_{\C{U}}^{(n)}\to
(\S_{\C{U}})^{n+1}$, which maps a diagram $\phi$ as in (a) to the
$(n+1)$-tuple $E^{(0)},\ldots,E^{(n)}$.
\end{Emp}

\begin{Emp} \label{E:remsn}
{\bf Remarks.} (a) Note that for every $X,Y\in sSp/K$, to give a
map $\phi\in\Hom_K(X,Y)$ is the same as to give maps
$\tau^*(\phi)\in\Hom_{\Box[n,m]}(\tau^*(X),\tau^*(Y))$ for all
$\phi:\Box[n,m]\to K$, compatible with compositions. Using this observation 
and \rl{univ}, we conclude that for every $K\in sSp$ we have a natural
bijection between $\Hom_{sSp}(K,\S_{\C{U}}^{(n)})$ and set of
diagrams $\phi:E^{(0)}\overset{\phi_1}{\lra} \ldots
\overset{\phi_n}{\lra}E^{(n)}$ of left fibrations from
$(sSp/K)_{\C{U}}$.

(b) By definition, a map $\phi\in\Hom_{sSp}(K,\S_{\C{U}}^{(1)})$
belongs to $\Hom_{sSp}(K,\S_{\C{U}}^{(we)})$ if and only if
$\phi(a)\in (\S_{\C{U}}^{(we)})_{n,m}$ for every $a\in K_{n,m}$.
Moreover, by \rl{leftequiv}, it happens if and only if $\phi(a)\in
(\S_{\C{U}}^{(we)})_{0,0}$ for every $a\in K_{0,0}$. Using \rl{leftequiv} again, 
we see that under the bijection of (a) elements of
$\Hom_{sSp}(K,\S_{\C{U}}^{(we)})\subset\Hom_{sSp}(K,\S_{\C{U}}^{(1)})$
correspond to weak equivalences $\phi:E^{(0)}\to E^{(1)}$.
\end{Emp}

The following two propositions and a corollary will be shown in
\rs{proofs}.

\begin{Prop} \label{P:spaces}
(a) The simplicial space $\S_{\C{U}}\in sSp$ is Ready fibrant.

(b) The projections $\S_{\C{U}}^{(n)}\to (\S_{\C{U}})^{n+1}$ and
$\S_{\C{U}}^{(we)}\to (\S_{\C{U}})^2$ are fibrations.

(c) Both compositions $\S_{\C{U}}^{(we)}\to
(\S_{\C{U}})^2\overset{p_i}{\lra} \S_{\C{U}}$ are trivial
fibrations.

(d) $(\S_{\C{U}}^{we})_0\subset (\S^{(1)}_{\C{U}})_0$ is a union
of a connected components.
\end{Prop}

\begin{Prop} \label{P:sn}
(a) There exists a homotopy equivalence $(\S_{\C{U}})^{\Dt[1]}\to
\S_{\C{U}}^{(we)}$ over $(\S_{\C{U}})^2$.

(b) For every $n\in\B{N}$ there exists a "natural" homotopy equivalence
$\psi^{(n)}_{\C{U}}:\S_{\C{U}}^{(n)}\to (\S_{\C{U}})^{F[n]}$ over
$(\S_{\C{U}})^{n+1}$, defined uniquely up to a homotopy.

(c) Moreover, for every map $\mu:[m]\to[n]$ the diagram
\begin{equation} \label{Eq:s}
\begin{CD}
 \S_{\C{U}}^{(n)} @>\psi^{(n)}>> (\S_{\C{U}})^{F[n]}\\
 @V\mu^*VV @V\mu^*VV\\
 \S_{\C{U}}^{(m)} @>\psi^{(m)}>> (\S_{\C{U}})^{F[m]}
 \end{CD}
\end{equation}
is homotopy commutative, that is, $\mu^*\circ\psi^{(n)}\sim
\psi^{(m)}\circ\mu^*$.
\end{Prop}


\begin{Cor} \label{C:presh}
Let $K\in sSp$, let $\al,\beta\in \Hom(K,\S_{\C{U}})$, and let
$E_{\al}\to K$ and $E_{\beta}\to K$ be the corresponding left
fibrations. Then $\al\sim\beta$ in $(\S_{\C{U}})^K$ if and only if
the left fibrations $E_{\al}$ and $E_{\beta}$ are homotopy
equivalent over $K$.
\end{Cor}


Now we are ready to prove one of the main results of this work.

\begin{Thm} \label{T:spaces}
$\S_{\C{U}}$ is a complete Segal space.
\end{Thm}

\begin{proof}
We denote $\S_{\C{U}}$ simply by $\S$. Then $\S$ is  fibrant by
\rp{spaces} (a).

To show that $\S$ is Segal, we have to prove that for every $n\geq
2$ the morphism $\varphi_n:\S_n\to\S_1\times_{\S_0}\ldots
\times_{\S_0} \S_1$ is a weak equivalence. Applying \rp{sn} (c) to
$(\dt^{01},\ldots,\dt^{n-1,n}):[n]\to [1]\times\ldots\times[1]$,
we get a homotopy commutative diagram
\[
\begin{CD}
\S^{(n)} @>\psi^{(n)}>> \S^{F[n]}\\
@VVV @VVV\\
\S^{(1)}\times_\S \ldots \times_\S\S^{(1)} @>\psi^{(1)}\times
\ldots\times\psi^{(1)}>> \S^{F[1]}\times_\S \ldots\times_\S
\S^{F[1]}.
 \end{CD}
\]
We want to show that the right vertical arrow is a weak
equivalence, which implies the Segal conditions by passing to the
zero spaces. The top horizontal arrow is a weak equivalence by
\rp{sn} (b). The bottom horizontal arrow is a equivalence by
\rp{sn} (b) together with the observation that $\S^{(1)}\to \S$
and $\S^{F[1]}\to \S$ are fibrations (use \rp{spaces} and
\rco{hfibprod}). Next, since the left vertical arrow is a
bijection, while diagram is homotopy commutative, the right
vertical arrow is a weak equivalence, by 2-out-of-3.

To show that $\S$ is complete, we have to show that
$\dt_0:\S_{heq}\to \S_0$ is a trivial fibration (by \rl{css} (b)).
Since $p_0:\S^{(we)}\to\S$ is a trivial fibration by \rp{spaces}
(c), it is enough to show that the map
$\psi:=(\psi^{(1)})_0:(\S^{(1)})_0\to (\S^{F[1]})_0=\S_1$ from
\rp{sn} (b) induces an equivalence $(\S^{(we)})_0\to\S_{heq}$.

Both $(\S^{(we)})_0\subset (\S^{(1)})_0$ and $\S_{heq}\subset\S_1$
are unions of connected components (by \rp{spaces} (d) and
\re{css} (b)). Since $\psi$ is a weak equivalence, it remains to
show that $\pi_0(\psi)$ induces a bijection
$\pi_0((\S^{(we)})_0)\to\pi_0(\S_{heq})$.

By \rp{sn} (c), we have the following homotopy commutative diagram
\begin{equation} \label{Eq:cdcss}
\begin{CD}
(\S^{(1)})_0 @<\dt_{12}<< (\S^{(3)})_0
@>\dt_{02},\dt_{13}>>(\S^{(1)}
\times \S^{(1)})_0  @<s_0\times s_0<<\S_0\times \S_0\\
@V\psi^{(1)}VV @V\psi^{(3)}VV @V\psi^{(1)}\times \psi^{(1)}VV @|\\
\S_1 @<\dt_{12}<< \S_3 @>\dt_{02},\dt_{13}>>\S_1 \times \S_1
@<s_0\times s_0<<\S_0\times \S_0.
\end{CD}
\end{equation}

Recall that in \rl{css} (a) we introduced unions of connected
components $\S'_1\subset\S_1$ and $\S'_3\subset\S_3$ and showed
that $\dt_{12}(\S'_3)=\S_{heq}$.

Similarly, we define $(\S^{(1)})'_0\subset (\S^{(1)})_0$ to be the
union of connected components, intersecting $s_0(\S_0)$, and set
$(\S^{(3)})'_0:=\dt_{02}^{-1}((\S^{(1)})'_0)\cap\dt_{13}^{-1}((\S^{(1)})'_0)
\subset(\S^{(3)})_0$.

We claim that $\dt_{12}((\S^{(3)})'_0)=(\S^{(we)})_0$. Indeed,
since $s_0(\S_0)\subset\S^{(we)}$, it follows from \rp{spaces} (d)
that for every $\wt{\phi}\in(\S^{(3)})'_0$ we have
$\dt_{02}(\wt{\phi}),\dt_{13}(\wt{\phi})\in \S^{(we)}$. In other
words, if $\wt{\phi}$ corresponds to a diagram
$E^{(0)}\overset{\phi_1}{\lra}E^{(1)}\overset{\phi_2}{\lra}E^{(2)}
\overset{\phi_3}{\lra}E^{(3)}$, then $\phi_2\circ\phi_1$ and
$\phi_3\circ\phi_2$ are weak equivalences. Therefore $\phi_2$ have
left and right homotopy inverses. Hence $\phi_2$ is a weak
equivalence, thus $\phi_2=\dt_{12}(\wt{\phi})\in \S^{(we)}$.

Conversely, every $\phi\in \S^{(we)}$ corresponds to a homotopy
equivalence $\phi:E^{(0)}\to E^{(1)}$ (by \rl{heprop} (b)), thus
there exists a diagram
$\wt{\phi}:E^{(1)}\overset{\phi'}{\lra}E^{(0)}\overset{\phi}{\lra}E^{(1)}
\overset{\phi'}{\lra}E^{(0)}$ such that
$\phi'\circ\phi\sim\Id_{E^{(0)}}$ and
$\phi\circ\phi'\sim\Id_{E^{(1)}}$. By definition, $\wt{\phi}$
corresponds to an element of $(\S^{(3)})'_0$ and
$\dt_{12}(\wt{\phi})=\phi$.

Now we are ready to show the assertion. Since $\psi^{(3)}$ is a
weak equivalence, the induced map
$\pi_0(\S^{(3)})_0)\to\pi_0(\S_3)$ is a bijection. Next, using
definitions of $(\S^{(3)})'_0$ and $\S'_3$ and the homotopy
commutativity of the interior and right inner squares of
\form{cdcss}, $\psi^{(3)}$ induces a bijection
$\pi_0((\S^{(3)})'_0)\to\pi_0(\S'_3)$. Finally, since
$\S_{heq}=\dt_{12}(\S'_3)$,
$(\S^{(we)})_0=\dt_{12}((\S^{(3)})'_0)$, and the left inner square
of \form{cdcss} is homotopy commutative, $\psi^{(1)}$ induces a
bijection $\pi_0(\S^{(we)})\to\pi_0(\S_{heq})$.
\end{proof}


%

\subsection{The Yoneda embedding}
\begin{Emp} \label{E:opp}
{\bf The opposite simplicial space.} (a) For every map 
$\tau:[n]\to[m]$ in $\Dt$, we denote by $\iota(\tau):[n]\to[m]$
the map $\iota(\tau)(n-i):=m-\tau(i)$. Then $\iota$ defines a
functor $\Dt\to\Dt$, hence a functor
$\iota^*:sSp=\Hom(\Dt^{op},Sp)\to \Hom(\Dt^{op},Sp)=sSp$.

(b) For every  $X\in sSp$, we set $X^{op}:=\iota^*(X)\in sSp$.
Explicitly, we have $(X^{op})_{n}=X_n$ for all $n$, and for every
$\tau:[n]\to[m]$ the map $\tau^*:(X^{op})_{m}\to (X^{op})_{n}$ is
the map $\iota(\tau)^*:X_m\to X_n$.

(c) Note that if $X$ is a (complete) Segal space, then $X^{op}$ 
is also such, and we have equality of homotopy categories $\Ho
(X^{op})=(\Ho X)^{op}$. Therefore we call $X^{op}$ {\em the
opposite simplicial space}.
\end{Emp}

\begin{Emp} \label{E:catmor}
{\bf The twisted arrow category.} (a) Consider the functor
$\mu:\Dt\to\Dt$ such that $\mu([n])=[2n+1]$, and for every
$\tau:[n]\to[m]$ in $\Dt$ the map $\mu(\tau):[2n+1]\to[2m+1]$ is
defined by formulas $\mu(\tau)(n-i):=m-\tau(i)$ and
$\mu(\tau)(n+1+j)=(m+1+\tau(j))$ for $i,j=0,\ldots,n$.

(b)  For every  $X\in sSp$, we define simplicial space
$\C{M}(X):=\mu^*(X)\in sSp$. Explicitly, we have
$\C{M}(X)_{n}=X_{2n+1}$ for all $n$, and for every $\tau:[n]\to[m]$
the map $\tau^*:\C{M}(X)_{m}\to \C{M}(X)_{n}$ is the map
$\mu(\tau)^*:X_{2m+1}\to X_{2n+1}$.

(c) We have natural morphisms $\iota\to \mu$ and  $\Id\to\mu$ of
functors $\Hom(\Dt,\Dt)$  which correspond to maps
$e^0:[n]\to[2n+1]$ and $e^{n+1}:[n]\to [2n+1]$, respectively. These 
maps corresponds to a morphism $\pi_X:\C{M}(X)\to X^{op}\times X$
in $sSp$.
\end{Emp}

\begin{Emp}
{\bf Remarks.} (a) Note that $\Dt$ is equivalent to the category
$\Dt'$ of finite totally ordered sets, and functors
$\iota,\mu:\Dt\to\Dt$ correspond to functors
$\iota,\mu:\Dt'\to\Dt'$ defined by $\iota(P)=P^{op}$ and
$\mu(P)=P^{op}*P$, the "join" of $P^{op}$ and $P$.

(b) It can be shown (using \rl{morp} and \re{left} (b)) that if $X$ is a 
(complete) Segal space, then $\C{M}(X)$ is a (complete) Segal space as well.
In this case, the space of objects $\C{M}(X)_0$ equals the space of morphisms $X_1$, 
and for every  $\al:x\to y$ and $\al':x'\to y'$ in $\C{M}(X)_0=X_1$, the mapping space
$\map_{\C{M}(X)}(\al,\al')$ can be intuitively thought as the space of triples $(\beta,\beta',\gm)$, 
where $\beta:x'\to x$ and $\beta':y\to y'$ belong to $X_1$, and $\gm$ is a path between 
$\beta'\circ\al\circ\beta$ and $\al'$. 
\end{Emp}

From now on in this subsection we always assume that $X$ is a
Segal space.

\begin{Lem} \label{L:morp}
The map $\pi_X:\C{M}(X)\to X^{op}\times X$ is a left fibration.
\end{Lem}

\begin{proof}
To show that $\pi_X$ is a fibration, we have to check that the
induced map
\begin{equation} \label{Eq:fib}
\C{M}(X)_n\to\C{M}(X)_{\p n}\times_{(X^{op}\times X)_{\p
n}}(X^{op}\times X)_n
\end{equation}
is a fibration for every $n\geq 0$. Recall that
$\C{M}(X)_n=X_{2n+1}=\Map(F[2n+1],X)$. Since $\p
F[n]=\cup_{i=0}^nd^iF[n-1]$ we get that $\C{M}(X)_{\p
n}=\Map(F[2n+1]',X)$, where $F[2n+1]'\subset F[2n+1]$ is the union
$\cup_{i=0}^n d^{i,2n+1-i}(F[2n-1])\subset F[2n+1]$. Thus morphism
\form{fib} can be identified with the morphism
\[
\Map(F[2n+1],X)\to \Map(F[2n+1]'\sqcup_{(e^0\partial F[n]\cup
e^{n+1} \p F[n])}(e^0 F[n]\cup e^{n+1} F[n]),X).
\]
Since $F[2n+1]'\cap(e^0 F[n]\cup e^{n+1} F[n])=e^0\partial
F[n]\cup e^{n+1}\p F[n]\subset F[2n+1]$, the natural map
$F[2n+1]'\sqcup_{(e^0\partial F[n]\cup e^{n+1} \p F[n])}(e^0
F[n]\cup e^{n+1} F[n])\to F[2n+1]$ 
is a cofibration. Since $X$ is fibrant, the map \form{fib} is a
fibration.

It remains to show that the fibration $\C{M}(X)_n\to
\C{M}(X)_0\times_{(X^{op}\times X)_0}(X^{op}\times X)_n$ or,
equivalently, $X_{2n+1}\to X_n\times_{X_0}X_1\times_{X_0} X_n$ is
a weak equivalence. Since $X$ is a Segal space, thus both the
composition
\[
X_{2n+1}\to X_n\times_{X_0}X_1\times_{X_0} X_n\to
(X_1\times_{X_0}\ldots\times_{X_0} X_1)
\times_{X_0}X_1\times_{X_0}(X_1\times_{X_0}\ldots\times_{X_0} X_1)
\]
and the second morphism are trivial fibrations, this 
follows from 2-out-of-3.
\end{proof}

\begin{Emp} \label{E:rem}
{\bf Remark.} It follows from \rl{morp} and \rl{left} (a) that
left fibration $\pi_X$ induces a left fibration
$\{x\}\times_{X^{op}}\C{M}(X)\to X$ for every $x\in X$. Notice that
$(\{x\}\times_{X^{op}}\C{M}(X))_0=\{x\}\times_{X_0}X_1=(x\bs X)_0$.
\end{Emp}

\begin{Lem} \label{L:we}
Let $X$ be a Segal space and $x\in X$. There exists a weak
equivalence $\wt{\phi}:x\bs X\to \{x\}\times_{X^{op}}\C{M}(X)$ of 
left fibrations over $X$ such that $\wt{\phi}(\id_x)\sim\id_x$.
\end{Lem}

\begin{proof}
It will suffice to construct a homotopy equivalence
$\psi:\{x\}\times_{X^{op}}\C{M}(X)\to x\bs X$ over $X$ such that
$\psi(\id_x)\sim \id_x$ and to take $\wt{\phi}$ to be its homotopy
inverse.

To construct $\psi$, we will construct a simplicial space
$\wt{x\bs X}$ over $X$ and maps $\psi':\wt{x\bs X}\to x\bs X$ and
$\psi'':\wt{x\bs X}\to\{x\}\times_{X^{op}}\C{M}(X)$ over $X$ such
that $(\wt{x\bs X})_0=(x\bs X)_0$, $\psi'_0=\psi'_0=\Id$ and
$\psi'$ is a trivial cofibration. Since $x\bs X\to X$ is a
fibration, the map $\psi''$ extends to a map
$\psi:\{x\}\times_{X^{op}}\C{M}(X)\to x\bs X$ over $X$.

In this case, $\psi_0=\Id$ would be a weak equivalence, so $\psi$
would be a weak equivalence by \rl{leftequiv}, hence a homotopy
equivalence by \rl{heprop} (b).

For every map $\tau:[n]\to[m]$, we denote by $\tau'$ the map $[n+1]\to[m+1]$ defined by 
$\tau'(0)=0$ and $\tau'(i+1)=\tau(i)+1$ for all $i=0,\ldots,n$.
Consider $\wt{x\bs X}\in sSp$ such that $(\wt{x\bs
X})_n:=\{x\}\times_{X_0,\dt_0}X_{n+1}$, and for every
$\tau$ the map $\tau^*:(\wt{x\bs X})_m\to(\wt{x\bs
X})_n$ is induced by  $\tau'^*:X_{m+1}\to X_{n+1}$. 

Note that projections $e_1:X_{n+1}\to X_n$ induce a projection $\wt{x\bs X}\to X$.
Next, map $[n]\times[1]\to[n+1]:(i,j)\mapsto(i+j)j$ induces maps 
$F[n]\times F[1]\to
F[n+1]$ and 
\[
p_n: X_{n+1}=\Map(F[n+1],X)\to \Map(F[n]\times F[1],X)=(X^{F[1]})_n.
\] 
Then $p_n$'s give rise to a map $\psi':\wt{x\bs X}\to x\bs X$ over $X$.

Finally, maps $r:[2n+1]\to[n+1]$, where $r(i)=0$ and
$r(i+n+1)=i+1$ for all $i=0,\ldots, n$ induce maps $X_{n+1}\to
X_{2n+1}$ and give rise to a map \\ $\psi'':\wt{x\bs X}\to
\{x\}\times_{X^{op}}\C{M}(X)$ over $X$, which we claim is a trivial
cofibration.

We have to show that $\psi''_n:\{x\}\times_{X_0}X_{n+1}\to
\{x\}\times_{X_n}X_{2n+1}$ is a trivial cofibration for all $n$. Since $X$ is
Segal, the natural map $X_{2n+1}\to X_n\times_{X_0}X_{n+1}$ is a
trivial fibration, whose pullback
$\pi_n:\{x\}\times_{X_n}X_{2n+1}\to\{x\}\times_{X_0}X_{n+1}$ is a
trivial fibration, satisfying $\pi_n\circ \psi''_n=\Id$. Therefore
$\psi''_n:\{x\}\times_{X_0}X_{n+1}\to \{x\}\times_{X_n}X_{2n+1}$
is a trivial cofibration by 2-out-of-3, and the proof is complete.
\end{proof}

\begin{Emp} \label{E:yoneda}
{\bf The Yoneda embedding.} We fix an infinite set $\C{U}$, and set
$\S:=\S_{\C{U}}$.

(a) For every $X\in sSp_{\C{U}}$, we set $\P(X):=\S^{X^{op}}$.
Since $\S$ is a complete Segal space, $\P(X)$ is also a complete
Segal space (by \re{CartSS}), and we call it the {\em
$\infty$-category of simplicial presheaves on $X$.}

(b) By definition of $\S$, for every Segal space $X$, the left
fibration $\pi_X:\C{M}(X)\to X^{op}\times X$ corresponds to the
morphism $X^{op}\times X\to \S$, hence to the morphism
$j_X:X\to\P(X)$. We call $j_X$ {\em the Yoneda embedding} of $X$.

(c) For every $x\in X$ and $\al\in\P(X)$, we form $\al(x)\in
\S$ and denote by $E_{\al(x)}\to\pt$ the corresponding left fibration.
\end{Emp}

\begin{Emp} \label{E:ex}
{\bf Example.} By definition, for every $x\in X$ element
$j_X(x)\in\P(X)=\S^{X^{op}}$ corresponds to the left fibration
$\C{M}(X)\times_X\{x\}=\{x\}\times_{X}\C{M}(X^{op})\to X^{op}$.
\end{Emp}

\begin{Thm} \label{T:yoneda}
Let $X$ be a Segal space, $x\in X$ and $\al\in\P(X)$.

(a) We have a "natural" weak equivalence
$\map_{\P(X)}(j_X(x),\al)\to (E_{\al(x)})_0$, canonical up to
homotopy.

(b) The Yoneda embedding $j_X:X\to\P(X)$ is fully faithful (see
\re{objmap} (d)).
\end{Thm}

First we have to introduce certain notation.

\begin{Emp} \label{E:wte}
{\bf The universal left fibration over $X^{op}$.} (a) Let
$\wt{E}\to X^{op}\times\P(X)$ be the left fibration, corresponding
to the evaluation map $\ev_{X^{op}}:X^{op}\times\P(X)\to \S$.
Then for every $\al\in\P(X)$, the pullback
$\wt{E}_{\al}:=\wt{E}\times_{\P(x)}\{\al\}$ is the left fibration
over $X^{op}$, corresponding to $\al$. In particular, for every
$x\in X$, its fiber $(\wt{E}_{\al})_x$ is the left fibration
$E_{\al(x)}\to\pt$, corresponding to $\al(x)\in \S$ (see
\re{yoneda} (c)).

(b) For every $x\in X^{op}$, we set
$\wt{E}_x:=\{x\}\times_{X^{op}}\wt{E}$. Then $\wt{E}_x\to\P(X)$ is
a left fibration such that $(\wt{E}_x)_{\al}=E_{\al(x)}$ for every
$\al\in\P(X)$ (by (a)). For every $y\in X$ we have an equality
$(\wt{E}_x)_{j_X(y)}=E_{j_X(y)(x)}$, thus
$((\wt{E}_x)_{j_X(y)})_0=\map_X(x,y)$. In particular, we have an
element $\id_x\in\map_X(x,x)$ of $(\wt{E}_x)_{j_X(x)}$.

(c) By definition, left fibration $\pi_X:\C{M}(X)\to X^{op}\times X$
corresponds to the composition of  $\Id\times j_X:X^{op}\times
X\to X^{op}\times\P(X)$ and $\ev_{X^{op}}$. Therefore we have an
equality $\C{M}(X)=\wt{E}\times_{\P(X)} X$, hence
$\wt{E}_x\times_{\P(X)}X=\{x\}\times_{X^{op}}\C{M}(X)$.
\end{Emp}

\begin{Emp} \label{E:notrem}
{\bf Remarks.} The homotopy equivalence $\psi^{(1)}:\S^{F[1]}\to
\S^{(1)}$ over $\S^2$ (see \rp{sn} (b)) induces a homotopy
equivalence from $\P(X)^{F[1]}=(\S^{F[1]})^{X^{op}}$ to
$\P(X)^{(1)}:=(\S^{(1)})^{X^{op}}$ over $\P(X)^2$. Hence for every
$\al\in \P(X)$ it induces a homotopy equivalence
$\psi_{\al}:\al\bs\P(X)\to
(\al\bs\P(X))':=\{\al\}\times_{\P(X)}\P(X)^{(1)}$ over $\P(X)$. Then, by
\rp{sn} (c),  we have $\psi_{\al}(\id_{\al})\sim\id_{\al}$. 

\end{Emp}


\begin{Lem} \label{L:equiv}
For every $x\in X$, there exists a weak equivalence
$\phi:j_X(x)\bs \P(X)\to \wt{E}_x$ of left fibrations over $\P(X)$
such that $\phi(\id_{j_X(x)})\sim\id_x\in (\wt{E}_x)_{j_X(x)}$.
\end{Lem}

\begin{proof}
We construct $\phi$ as a composition of weak equivalences over
$\P(X)$
\[
j_X(x)\bs\P(X)\overset{\phi'''}{\lra}(j_X(x)\bs\P(X))'
\overset{\phi''}{\lra}((x\bs X^{op})\bs
\P(X))'\overset{\phi'}{\lra}\wt{E}_x.
\]

By definition, if $\al:K\to\P(X)$ corresponds to the left
fibration $G\to X^{op}\times K$, then maps $K\to((x\bs X^{op})\bs
\P(X))'$ over $\al$ are in bijection with maps $\nu:(x\bs
X^{op})\times K\to G$ over $X^{op}\times K$, and maps
$K\to(j_X(x)\bs \P(X))'$ over $\al$ are in bijection with maps
$\nu':(\{x\}\times_X\C{M}(X^{op}))\times K\to G$ over $X^{op}\times
K$ (use \re{ex}).

Let $\wt{\phi}:x\bs X^{op}\to \{x\}\times_X\Mor(X^{op})$ be the
weak equivalence from \rl{we}, and we define $\phi''$ to be the
map, which sends $\nu'$ to $\nu'\circ\wt{\phi}$. Then $\phi''$ is
a homotopy equivalence, because $\wt{\phi}$ is such.

Next we observe that maps $K\to\wt{E}_x$ over $\al:K\to\P(X)$ are in
bijection with sections $s$ of the left fibration
$G_x:=\{x\}\times_{X^{op}}G\to K$. We define $\phi'$ to be the
map, which sends $\nu:(x\bs X^{op})\times K\to G$ to
$s:=\nu|_{\id_x}:K\to G_x$.

Since $\phi'$ is a map between left fibrations, to show that
$\phi'$ is a weak equivalence, it remans to show that for every
$\al\in\P(X)$ the induced map $(\phi'_{\al})_0$ is a weak
equivalence of simplicial sets (by \rl{leftequiv}). Let $G\to
X^{op}$ be the left fibration corresponding to $\al$. Then
$(\phi'_{\al})_0$ is the trivial fibration $\ev_{\id_x}:\Map_{X^{op}}(x\bs
X^{op},G)\to (G_x)_0$ from \rp{triv}.

Finally, we define $\phi''':j_X(x)\bs\P(X)\to(j_X(x)\bs\P(X))'$ to
be the weak equivalence from \re{notrem}, and set
$\phi:=\phi'\circ\phi''\circ\phi'''$. By construction, we have
$\phi'''(\id_{j_X(x)})\sim\id_{j_X(x)}$,
$\phi''(\id_{j_X(x)})=\wt{\phi}$, and
$\phi'(\wt{\phi})=\wt{\phi}(\id_x)\sim\id_x$. Thus
$\phi(\id_{j_X(x)})\sim\id_x$.
\end{proof}

Now we are ready to prove the Yoneda lemma.

\begin{proof}[Proof of \rt{yoneda}]

(a) By \rl{equiv}, there exists a homotopy equivalence
$\phi:j_X(x)\bs \P(X)\to \wt{E}_x$ of left fibrations over $\P(X)$
such that $\phi(\id_{j_X(x)})\sim\id_x$. Since for every
$\al\in\P(X)$ the fiber of $\wt{E}_x$ at $\al$ is $E_{\al(x)}$
(see \re{wte} (b)), $\phi$ induces an equivalence
$\phi_{\al}:\map_{\P(X)}(j_X(x),\al)\to (E_{\al(x)})_0$.

(b) The morphism $j_X:X\to \P(X)$ induces a morphism $x\bs X\to
j_X(x)\bs \P(X)$ over $j_X$.  Hence $j_X$ induces a morphism
$\psi:x\bs X\to j_X(x)\bs\P(X)\times_{\P(X)}X$ of left fibrations
over $X$. We claim that $\psi$ is a weak equivalence, hence it
induces a weak equivalence of fibers
$\psi_b:\map_X(a,b)\to\map_{\P(X)}(j_X(a),j_X(b))$ for all $b\in
X$.

Consider the composition $\phi\circ\psi:x\bs X\to
\wt{E}_x\times_{\P(X)}X$, where $\phi$ is as in the proof of (a).
Since $\wt{E}_x\times_{\P(X)}X=\{x\}\times_{X^{op}}\C{M}(X)$ (see
\re{wte} (c)), $\phi\circ\psi$ is a map $x\bs X\to
\{x\}\times_{X^{op}}\C{M}(X)$ over $X$, which by construction
satisfy $\phi\circ\psi(\id_x)\sim\id_x$. Therefore $\phi\circ\psi$
is a weak equivalence by \rco{trfib} and \rl{we}, hence $\psi$ is
a weak equivalence by 2-out-of-3.
\end{proof}

\section{Quasifibrations of simplicial spaces}
\subsection{Definitions and basic properties}
\begin{Def} \label{D:qfib}
Let ${\C{C}}$ be a right proper model category. We say that a
morphism $p:X\to B$ in ${\C{C}}$ is a {\em quasifibration}, if for
every weak equivalence $g:Y\to Z$ over $B$, its pullback
$p^*(g):p^*(Y)\to p^*(Z)$ (see \re{rlp} (d)) is a weak equivalence over $X$.
\end{Def}

\begin{Emp} \label{E:qfib}
{\bf Remarks.} (a) By definition, any pullback of a quasifibration
is a quasifibration, and a composition of quasifibrations is a
quasifibration.

(b) When ${\C{C}}$ is a model category of topological
spaces, our notion of a quasifibration is stronger than the classical notion. However
we will only use this notion for model categories of spaces and
simplicial spaces, where no classical notion exists.

(c) After this work was essentially completed, we found that quasifibrations were also studied
in a unpublished preprint of Rezk \cite{Re3} under a name of {\em sharp} morphisms. But we think that our
terminology is more suggestive.
\end{Emp}

\begin{Lem} \label{L:qfib}
(a) Every fibration in ${\C{C}}$ is a quasifibration.

(b) Let $f:X\to X'$ be a weak equivalence between quasifibrations
$p:X\to B$ and $p':X'\to B$ over $B$.  Then  for every morphism
$\tau:A\to B$ the pullback $\tau^*(f):X\times_B A\to X'\times_B A$
is a weak equivalence.

(c) Conversely, assume that $f:X\to X'$ is a weak equivalence over
$B$ such that each pullback $\tau^*(f)$ is a weak equivalence.
Then $p:X\to B$ is a quasifibration if and only if $p':X'\to B$ is
a quasifibration.
\end{Lem}

\begin{proof}
(a) follows from the fact that ${\C{C}}$ is right proper.

(b)  If $\tau:A\to B$ is fibration, then $\tau^*(g)$ is a weak
equivalence by (a). If $\tau$ is a weak equivalence, then
$\tau^*(X)\to X$ and $\tau^*(X')\to X'$ are weak equivalences,
because $X\to B$ and $X'\to B$ are quasifibrations. Therefore
$\tau^*(f):\tau^*(X)\to\tau^*(X')$ is a weak equivalence by 
2-out-of-3. Since every morphism decomposes as a
composition of a trivial cofibration and a fibration, the general
case follows.

(c) Let $g:Y\to Z$ be a weak equivalence over $B$. Then
$X\times_B Y\to X'\times_B Y$ and $X\times_B
Z\to X'\times_B Z$ are weak equivalences by the assumption on $f$.
Hence, by 2-out-of-3, $p'^*(g): X'\times_B Y\to
X'\times_B Z$ is a weak equivalence if and only if $p^*(g)$ is
a weak equivalence. Thus, by definition, $p:X\to B$ is a
quasifibration if and only if $p':X'\to B$ is a quasifibration.
\end{proof}

\begin{Cor} \label{C:fibrep}
Assume that model category ${\C{C}}$ has the property that a
pullback of a cofibration is a cofibration, and let $p':X'\to B$
be a fibrant replacement of $p:X\to B$. Then $p$ is a
quasifibration if and only if $\tau^*(p')$ is a fibrant
replacement of $\tau^*(p)$ for every $\tau:A\to B$.
\end{Cor}

\begin{proof}
By MC5, $p$ decomposes as $X\overset{i}{\lra}X'\overset{p'}{\lra}B$, 
where $i$ is a trivial cofibration, and $p$ is a fibration. 
Then for every map $\tau:A\to B$, $\tau^*(p')$ is a fibration, while $\tau^*(i)$ is 
a cofibration. Thus we have to show that $X\to B$
is a quasifibration if and only if each $\tau^*(i)$ is a weak
equivalence. Since $p':X'\to B$ is a quasifibration by \rl{qfib}
(a), the assertion follows from \rl{qfib} (b) and (c).
\end{proof}


\begin{Cor} \label{C:hfibprod}
Suppose we are given a commutative diagram
\[
\begin{CD}
X' @>>>Z' @<g'<< Y'\\
@VVV @VVV @VVV\\
X @>>>Z @<g<< Y,
\end{CD}
\]
where $g$ and $g'$ are quasifibrations and all vertical morphisms
are weak equivalences. Then the induced map $X'\times_{Z'} Y'\to
X\times_Y Z$ is a weak equivalence.
\end{Cor}

\begin{proof}
Weak equivalence $Y'\to Y$ decomposes as composition $Y'\to
Z'\times_{Z} Y\to Y$, the second on which is a weak equivalence,
because it is a pullback of a weak equivalence $Z'\to Z$ along a
quasifibration $Y\to Z$. Therefore by 2-out-of-3 
$Y'\to Z'\times_{Z} Y$ is a weak equivalence between
quasifibrations over $Z'$.

Now map $X'\times_{Z'} Y'\to X\times_Y Z$ decomposes as
composition
\[
X'\times_{Z'} Y'\to X'\times_{Z'}(Z'\times_{Z} Y)=X'\times_{Z}
Y\to  X\times_Z Y,
\]
the first of which is a weak equivalence, being a pullback
of a weak equivalence between quasifibrations (use \rl{qfib} (b)),
while the second one is a weak equivalence, since $Y\to Z$ is a
quasifibration, and $X'\to X$  is a weak equivalence.
\end{proof}



From now on we assume that ${\C{C}}$ is
ether category $Sp$ with Kan model structure (\rt{Kan}) or category $sSp$
with Reedy model structure (\rt{Reedy}). 

\begin{Lem} \label{L:pushout}
Let $X_B\to B$ be a quasifibration in $sSp$, $A\to B$ a
cofibration, and $i:X_A:=X_B\times_B A\to Y_A$ a weak equivalence
of quasifibrations over $A$. Then the pushout
$Y_B:=X_B\sqcup_{X_A}Y_A$ is a quasifibration over $B\sqcup_A
A=B$, and the natural map $j:X_B\to Y_B$ is a weak equivalence.
\end{Lem}

\begin{proof}
By \rl{qfib} (c), it is enough to show that for every map
$\tau:B'\to B$, the pullback $\tau^*(j)$ is a weak equivalence.

Note that $\tau^*(j)$ is a pushout of $\tau^*(i):\tau^*(X_A)\to
\tau^*(Y_A)$ along $\tau^*(X_A)\to \tau^*(X_B)$. Since $A\to B$ a
cofibration, the induced maps $X_A\to X_B$ and $\tau^*(X_A)\to
\tau^*(X_B)$ are cofibrations. Since $i$ is a weak equivalence
between quasifibrations, $\tau^*(i)$ is a weak equivalence by
\rl{qfib} (b). Hence the pushout $\tau^*(j)$ is a weak
equivalence, because $sSp$ is left proper.
\end{proof}

\begin{Lem} \label{L:qfib0}
A morphism $p:X\to B$ is a quasifibration in $sSp$ if and only if
$p_n:X_n\to B_n$ is a quasifibration in $Sp$ for every $n$.
\end{Lem}

\begin{proof}
Note that if $p_n$ is a quasifibration for all $n$, then for every
weak equivalence $g:Y\to Z$ over $B$, the corresponding maps
$g_n:Y_n\to Z_n$ are weak equivalences over $B_n$ for all $n$.
Therefore each $p^*(g)_n=p_n^*(g_n)$ is a weak equivalence, since
$p_n$ is a quasifibration. Hence $p^*(g)$ is a weak equivalence,
thus $p$ is a quasifibration.

Conversely, assume that $p$ is a quasifibration. Every weak
equivalence $\wt{g}:\wt{Y}\to \wt{Z}$ over $B_n$ defines a weak
equivalence $g:=\wt{g}\times\Id_{F[n]}:\wt{Y}\times F[n]\to
\wt{Z}\times F[n]$ over $B\times F[n]$ such that $\wt{g}$ is the
restriction of $g_n$ to $\Id_{F[n]}\in F[n]_n$. Since $p$ is a
quasifibration, the pullback $p^*(g)$ is a weak equivalence over
$F[n]$. Thus $p_n^*(\wt{g})=(p^*(g)_n)_{\Id_{F[n]}}$ is a weak
equivalence, implying that $p_n$ is a quasifibration.
\end{proof}

\begin{Def} \label{D:lqfib}
A map $f:X\to B$ in $sSp$ is called a {\em left quasifibration},
if it is quasifibration and the morphism $X_n\to X_0\times_{B_0}
B_n$, induced by the inclusion $\dt^0:[0]\hra[n]$, is a weak
equivalence for all $n$.
\end{Def}

\begin{Emp} \label{E:remlqfib}
{\bf Remark.} It follows from \rl{leftcart} (c) that a fibration
$f:X\to B$ in $sSp$ is a left fibration if and only if it is a
left quasifibration.
\end{Emp}

\begin{Lem} \label{L:lqf}
Suppose we have a commutative diagram
\begin{equation} \label{Eq:lqfib}
\begin{CD}
X' @ >g>> X\\
@ Vf'VV @VfVV\\
B' @>h>> B
\end{CD}
\end{equation}
in $sSp$, where $g$ and $h$  are weak equivalences, while $f$ and
$f'$ are quasifibrations. Then $f$ is a left if and only if $f'$
is left.
\end{Lem}

\begin{proof}
Consider commutative diagram induced by \form{lqfib}
\[
\begin{CD}
X'_n @ >g_n>> X_n\\
@ Vp'_nVV @Vp_nVV\\
X'_0\times_{B'_0} B_n' @>g_0\times_{h_0} h_n>>X_0\times_{B_0} B_n. 
\end{CD}
\]
We have to show that $p_n$  is a weak
equivalence if and only if $p'_n$ is. Since $g$ is a weak
equivalence, by 2-out-of-3, it suffices to show
that $g_0\times_{h_0}h_n$ is a weak equivalence.

By \rl{qfib0}, maps $X_0\to B_0$ and $X'_0\to B'_0$ are
quasifibrations. Therefore $g_0\times_{h_0}h_n$ is a weak
equivalence by assumption and \rco{hfibprod}.
\end{proof}

\begin{Cor} \label{C:lqf}
A fibrant replacement of a left quasifibration is a left fibration.
\end{Cor}

\begin{proof}
Apply \rl{lqf} and remark \re{remlqfib} to the case when $B'=B$
and $f$ is a fibrant replacement of $f'$. 
\end{proof}

\subsection{Extension of fibrations, and fibrant replacements}

\begin{Lem} \label{L:fib}
Suppose we are given a commutative diagram
\[
\begin{CD}
 Y_A @>>>X_B\\
@ VVV @VVV\\
A @>>> B
\end{CD}
\]
in $sSp$  such that vertical arrows are fibrations, horizontal
arrows are cofibration, and the induced map $Y_A\to
X_A:=X_B\times_B A$ is a trivial cofibration.

Then there exists the largest simplicial subspace $Y_B\subset X_B$
such that $Y_B\times_B A=Y_A$. Moreover, $Y_B\subset X_B$ is a
strong deformation retract over $B$. In particular, $Y_B\to B$ is
a fibration, and the inclusion $i:Y_B\hra X_B$ is a trivial
cofibration.
\end{Lem}

\begin{proof}
Consider simplicial subspace $Y_B\subset X_B$ such that
$(Y_B)_{n,m}$ is the set of all $\tau\in
(X_B)_{n,m}=\Hom(\Box[n,m],X_B)$ such that
$\tau(\tau^{-1}(X_A))\subset Y_A$. Then $Y_B\subset X_B$ is
the largest subspace such that $Y_B\times_B A=Y_A$. By
construction, for every morphism $C\to B$, the set $\Hom_B(C,Y_B)$
can be identified with the set of maps $f\in\Hom_B(C,X_B)$ such
that $f(C\times_B A)\subset Y_A$.

Since $Y_A\subset X_A$ is a trivial cofibration between fibrations
over $A$, it is a strong deformation retract (see \re{sdr} (c)).
Thus there exists a map $g:X_A\times \Dt[1]\to X_A$  over $A$ such
that (i) $g|_{X_A\times\{0\}}=\Id_{X_A}$; (ii)
$g(X_A\times\{1\})\subset Y_A$ and (iii)
$g|_{Y_A\times\Dt[1]}=\pr_1:Y_A\times\Dt[1]\to Y_A\subset X_A$.

Since $Y_B\cap X_A=Y_A$, property (iii) of $g$ implies that $g$
extends to a map $g':(X_A\sqcup_{Y_A}Y_B)\times\Dt[1]\to X_B$ over
$B$ such that $g'|_{Y_B\times\Dt[1]}=\pr_1$. Next property (i) of
$g$ implies that $g'$ extends to a morphism
\[
g'':((X_A\sqcup_{Y_A}Y_B)\times\Dt[1])
\sqcup_{(X_A\sqcup_{Y_A}Y_B)\times\{0\}}(X_B\times\{0\}) \to X_B
\]
over $B$ such that $g''|_{X_B\times \{0\}}=\Id_{X_B}$.

Since $\{0\}\hra\Dt[1]$ is a trivial cofibration, while model
category $sSp$ is Cartesian, we conclude that
\[
((X_A\sqcup_{Y_A}Y_B)\times\Dt[1])\sqcup_{(X_A\sqcup_{Y_A}Y_B)\times\{0\}}
(X_B\times \{0\})\hra X_B\times\Dt[1]
\]
is a trivial cofibration, thus $g''$ extends to a map $h:X_B\times
\Dt[1]\to X_B$ over $B$.

Then $h$ satisfies 
$h|_{X_B\times\{0\}}=g''|_{X_B\times\{0\}}=\Id_{X_B}$,
$h|_{Y_B\times \Dt[1]}=g'|_{Y_B\times \Dt[1]}=\pr_1$ and
$h(X_A\times\{1\})=g(X_A\times\{1\})\subset Y_A$. Thus, by the
construction of $Y_B$, we get that $h(X_B\times\{1\})\subset Y_B$.
In other words, $h$ realizes $Y_B$ as a strong deformation retract
of $X_B$ over $B$. The last assertion follows from \re{sdr} (b).
\end{proof}

\begin{Cor} \label{C:almost1}
Let $X\to B$ be an quasifibration, and let $A\hra B$ be a
cofibration such that $X\times_B A\to A$ is a fibration. Then
there exists a fibrant replacement $Y\to B$ of $X\to B$ such that
$X\times_B A=Y\times_B A$.
\end{Cor}

\begin{proof}
Let $X\overset{i}{\lra} Y'\overset{p}{\lra} B$ be any
decomposition of $X\to B$, where $i$ is a trivial cofibration, and
$p$ is a fibration. Since $X\to B$ is quasifibration, $i$ induces
a trivial cofibration $X\times_B A\hra  Y'\times_B A$ over $A$ (by
\rl{qfib} (b)).

Let $Y\subset Y'$ be the largest simplicial subspace such that
$Y\times_B A= X\times_B A$. Then $Y\to B$ is a fibration, and 
$Y\hra Y'$ is a trivial cofibration (see \rl{fib}). Since trivial
cofibration $i$ factors as  a composition $X\hra Y\hra Y'$, the
map $X\hra Y$ is a trivial cofibration by 2-out-of-3.
\end{proof}

\begin{Emp} \label{E:pfhepropd}
\begin{proof}[Proof of \rl{heprop} (d)]
Composition $Y_A\to A\to B$ decomposes as a composition of a
trivial cofibration $Y_A\to X_B$ and a fibration $X_B\to B$. Now
let $Y_B\subset X_B$ be the largest simplicial subspace such that
$Y_B\times_B A=Y_A$. Then $Y_B\to B$ is a fibration by \rl{fib}.
\end{proof}
\end{Emp}

\begin{Emp} \label{E:skel2}
{\bf Notation.} (a) For $X\in sSp$ and $n\geq 0$, we define
the $n$-skeleton $sk_n X$ to be the smallest simplicial subspace
$Y\subset X$ such that $Y_{m,k}=X_{m,k}$ for all $m+k\leq n$. Then
$0=sk_{-1} X\subset sk_0 X\subset \ldots \subset sk_n
X\subset\ldots \subset X$ and $X=\colim_n sk_n X$.

(b) For every $m,k\geq 0$, we denote by $X_{m,k}^{nd}:=X_{m,k}\sm
(sk_{m+k-1}X)_{m,k}$ the set of "non-degenerate bisimplices". Then
$X_{0,0}^{nd}=X_{0,0}$, and for each $n>0$ the $n$-th skeleton
$sk_n X$ is naturally isomorphic to the pushout of  $sk_{n-1} X$
and $\sqcup_{m+k=n, a\in X_{m,k}^{nd}}\Box[m,k]$ over
$\sqcup_{m+k=n, a\in X_{m,k}^{nd}}\p\Box[m,k]$.

\end{Emp}

\begin{Lem} \label{L:fibrepl}
Let $f:X\to K$ be a morphism in $(sSp/K)_{\C{U}}$ such that either
(a) $K\in sSp_{|\C{U}|}$ or (b) $f$ is a quasifibration. Then $f$
has a fibrant replacement  $f':X'\to K$ in $(sSp/K)_{\C{U}}$.
\end{Lem}

\begin{proof}
(a) Since $(sSp/K)_{\C{U}}\to (sSp/K)_{|\C{U}|}$ is an equivalence
of categories, it is enough to show the existence of a fibrant
replacement $f'$ in $(sSp/K)_{|\C{U}|}$. Since $K\in
sSp_{|\C{U}|}$ and $|\C{U}|\times |\C{U}|=|\C{U}|$, we conclude
that $X\in sSp_{|\C{U}|}$. Moreover, using $|\C{U}|\times
|\C{U}|=|\C{U}|$ again, we get that fibrant replacement $f':X'\to
K$, constructed in the proof of \rt{Reedy}, satisfies $X'\in
sSp_{|\C{U}|}$, thus $f'\in (sSp/K)_{|\C{U}|}$.

(b) By induction on $i$, we are going to construct a fibrant
replacement $f'[i]:X'[i]\to sk_i K$ of $f|_{sk_i K}:X|_{sk_i K}\to
sk_i K$ such that $f'[i]$ belongs to $(sSp/sk_i K)_{\C{U}}$ and
$f'[i+1]|_{sk_i K}=f'[i]$.

Assuming this is done, we set $f':=\colim_i f'[i]:X'\to K$. Then
$f'$ satisfies $f'|_{sk_i K}=f'[i]$, thus $f'$ is a fibration and
$f'\in (sSp/K)_{\C{U}}$. Moreover, since $X|_{sk_i K}\hra X'[i]$
is a weak equivalence for all $i$, we get that $X\hra X'$ is a
weak equivalence as well, thus $f'$ is a fibrant replacement of
$f$.

Assume that $f'[i]$ was already constructed. Set
$X[i]:=X\sqcup_{X|_{sk_i K}}X'[i]$. Then it follows from
\rl{pushout} that $f[i]:X[i]\to K$ is a quasifibration and that
$X\to X[i]$ is a weak equivalence. By construction,  $f[i]|_{sk_i
K}:X[i]|_{sk_i K}\to sk_i K$ is a fibration $f'[i]$, and we want
to show that $f[i]|_{sk_{i+1} K}$ has a fibrant replacement
$f'[i+1]$, all of whose fibers are in $Set_{\C{U}}$, such that
$f'[i+1]|_{sk_i K}=f[i]|_{sk_i K}$.

Recall that $sk_{i+1} K=sk_i K\sqcup_{\sqcup {\p\Box[n,m]}}{\sqcup
{\Box[n,m]}}$. Thus it remains to show that for every quasifibration 
$f:X\to\Box[n,m]$ from $(sSp/\Box[n,m])_{\C{U}}$, whose restriction to $\p\Box[n,m]$
is a fibration, there exists a fibrant replacement
$f':X'\to\Box[n,m]$ from $(sSp/\Box[n,m])_{\C{U}}$ such that
$f'|_{\p\Box[n,m]}=f|_{\p\Box[n,m]}$.

Since $\Box[n,m]\in sSp_{|\C{U}|}$, it follows from (a) that there
exists a fibrant replacement $f'':X''\to\Box[n,m]$ of $f$ in
$(sSp/\Box[n,m])_{\C{U}}$. Moreover, since $f$ is a
quasifibration, there exists $X'\subset X''$ such that
$X'|_{\p\Box[n,m]}=X|_{\p\Box[n,m]}$, and $f':=f''|_{X'}$ is a
fibrant replacement of $f$ (see the proof of \rco{almost1}). Then
 $f'\in (sSp/\Box[n,m])_{\C{U}}$ and $f'|_{\p\Box[n,m]}=f|_{\p\Box[n,m]}$.
\end{proof}

\begin{Cor} \label{C:almost}
In the situation \rco{almost1}, assume that the quasifibration
$X\to B$ is in $(sSp/B)_{\C{U}}$. Then $Y\to B$ can also chosen to
be in $(sSp/B)_{\C{U}}$.
\end{Cor}

\begin{proof}
If $X\to B$ is in $(sSp/B)_{\C{U}}$, then $Y'\to B$ from the proof
of \rco{almost1} can be chosen to be in $(sSp/B)_{\C{U}}$ by
\rl{fibrepl}, thus $Y\to B$ is also in $(sSp/B)_{\C{U}}$.
\end{proof}


\section{Complements} \label{S:proofs} 
\subsection{Discrete iterated cylinder}

\begin{Emp} \label{E:fn}
{\bf Observations.}  (a) Every map $p:X\to B\times F[m]$ in $sSp$
induces a map $p_n:X_n\to B_n\times F[m]_n$ for all $n$. Since
$F[m]_n$ decomposes as $\coprod_{\tau:[n]\to[m]}\pt$,
each $X_n$ decomposes as a disjoint union
$X_n=\coprod_{\tau:[n]\to[m]}X_{\tau}$, and $p_n$ decomposes as a
disjoint union of $p_{\tau}:X_{\tau}\to B_n$.

(b) By definition, all fibers of $p$ belong to $Set_{\C{U}}$ if
and only of all fibers of each $p_{\tau}$ belong to $Set_{\C{U}}$.

(c) Notice that $p_n$ is a quasifibration if and only if each
$p_{\tau}:X_{\tau}\to B_n$ is a quasifibration. Using \rl{qfib0}
we conclude that $p$ is a quasifibration if and only if each
$p_{\tau}$ is a quasifibration.

(d) Recall that a quasifibration  $p:X\to B\times F[m]$ is left if
and only if each morphism  $g_n:X_n\to X_0\times_{(B\times
F[m])_0}(B\times F[m])_n$ is a weak equivalence. Note that $g_n$
decomposes as a disjoint union of morphisms $g_{\tau}:X_{\tau}\to
X_{\tau|_0}\times_{B_0}B_n$. Thus $f$ is left if and only if each
$g_{\tau}$ is a weak equivalence.

(e) A morphism $f:X\to Y$ between left quasifibrations over
$B\times F[m]$ is a weak equivalence if and only if $f|_i:X|_i\to
Y|_i$ is a weak equivalence over $B$ for each $i=1\ldots,m$.

\begin{proof}
By definition, $f$ is a weak equivalence if and only if
$f_n:X_n\to Y_n$ is a weak equivalence for all $n$, which by (a)
is equivalent to the assertion that $f_{\tau}:X_{\tau}\to
Y_{\tau}$ is a weak equivalence for all $\tau:[n]\to[m]$.
Similarly, $f|_i$ is a weak equivalence if and only if
$f_{\tau}:X_{\tau}\to Y_{\tau}$ is a weak equivalence for all
$\tau:[n]\to\{i\}\subset [m]$. This implies "the only if"
assertion.

To show the converse, notice that by (d) and 2-out-of-3,
$f_{\tau}$ is a weak equivalence if and only if 
$f_{\tau|_0}:X_{\tau|_0}\times_{B_0}B_n\to
Y_{\tau|_0}\times_{B_0}B_n$ is a weak equivalence. Since by (b)
the maps $X_{\tau|_0}\to B_0$ and $Y_{\tau|_0}\to B_0$ are
quasifibrations, if follows from \rco{hfibprod} that $f$ is a weak
equivalence if each $f_{\tau|_0}:X_{\tau|_0}\to Y_{\tau|_0}$ is a
weak equivalence or, equivalently, if $f_{\tau}:X_{\tau}\to
Y_{\tau}$ is a weak equivalence for all $\tau:[0]\to[m]$.
\end{proof}
\end{Emp}

\begin{Emp} \label{E:itercyl}
{\bf Discrete iterated cylinder.} For $B\in sSp$ and a sequence of
morphisms $f: K^{(0)}\overset{f_1}{\to} \ldots
\overset{f_{m}}{\to}K^{(m)}$ in $sSp/B$, we define recursively a
{\em discrete iterated cylinder} $Cyl^{disc}(f)\to B\times F[m]$
and morphisms $\iota_j:K^{(j)}\times e^j F[m-j]\to Cyl^{disc}(f)$
over $B\times F[m]$ for all $j=0,\ldots, m$ as follows.

If $m=0$, we set $Cyl^{disc}(f):=K^{(0)}$, and put $\iota_0=\Id$.
If $m\geq 1$, we denote by $f(1)$ the sequence $f(1):
K^{(1)}\overset{f_2}{\to}\ldots \overset{f_{m}}{\to}K^{(m)}$, and
assume by induction that we have defined an iterated cone
$Cyl^{disc}(f(1))\to B\times F[m-1]$ and a morphism $\iota_j:
K^{(j)}\times e^{j}F[m-j]\to e^1 Cyl^{disc}(f(1))$ for all
$j=1,\ldots,m$.

We define $Cyl^{disc}(f)\to B\times F[m]$ to be the pushout
\begin{equation} \label{Eq:itercyl}
Cyl^{disc}(f):=(K^{(0)}\times F[m])\sqcup_{(K^{(0)}\times e^1
F[m-1])}e^1 Cyl^{disc}(f(1)),
\end{equation}
where the map  $K^{(0)}\times e^1 F[m-1]\to e^1 Cyl^{disc}(f(1))$
is defined to be the composition
\[
K^{(0)}\times e^1 F[m-1]\overset{f_1}{\lra} K^{(1)}\times e^1
F[m-1]\overset{\iota_1}{\lra}e^1 Cyl^{disc}(f(1)).
\]

Finally, we define $\iota_0:K^{(0)}\times F[m]\hra Cyl^{disc}(f)$
be the natural embedding.
\end{Emp}

\begin{Lem} \label{L:disc}
If each $K^{(i)}\to B$ is a (left) quasifibration in
$(sSp/B)_{\C{U}}$, then  $Cyl^{disc}(f)\to B\times F[m]$ is a
(left) quasifibration in $(sSp/B\times F[m])_{\C{U}}$.
\end{Lem}

\begin{proof}
We are going to apply \re{fn} to the  projection $p:Cyl^{disc}(f)\to
B\times F[m]$.

We claim that for each $\tau:[n]\to[m]$, we have
$Cyl^{disc}(f)_{\tau}=K^{(\tau(0))}_{n}$. The proof goes by
induction. The assertion is obvious, if $m=0$. Next, if $m\geq 1$,
then \form{itercyl} implies that $Cyl^{disc}(f)_{\tau}$ equals
$K^{(0)}_n$, if $\tau(0)=0$, and equals $Cyl^{disc}(f(1))_{\tau'}$,
where $\tau':[n]\to [m-1]$ is given by $\tau'(i)=\tau(i)-1$, if
$\tau(0)\geq 1$. By induction hypothesis, in the second case
$Cyl^{disc}(f)_{\tau}$ equals
$K^{(\tau'(1))}_{n}=K^{(\tau(0))}_{n}$.

By the proven above and \re{fn} (a), each fiber of
$p$ is a fiber of some
$K^{(\tau(0))}_{n}\to B_n$. Thus it belongs to $Set_{\C{U}}$, because
$K^{(\tau(0))}\to B$ is in $(sSp/B)_{\C{U}}$.

Moreover, since each $K^{(\tau(0))}\to B$ is a quasifibration,
each $K^{(\tau(0))}_n\to B_n$ is a quasifibration (by \rl{qfib0}).
Therefore $p$ is a quasifibration by \re{fn} (c). Similarly, each projection
$Cyl^{disc}(f)_{\tau}\to Cyl^{disc}(f)_{\tau|_0}\times_{B_0}B_n$
is simply $K^{(\tau(0))}_n\to K_0^{(\tau(0))}\times_{B_0}B_n$.
Therefore it is a weak equivalence, because $K^{(\tau(0))}\to B$
is left. The assertion now follows from  \re{fn}
(d).
\end{proof}

\begin{Lem} \label{L:lfibr}
Let $f:X\to Y$ be a map, and $E\to Y\times F[n]$ be a left
fibration.

(a) The map $p:\C{Map}_{Y\times F[n]}(X\times F[n],E)\to
\C{Map}_{Y\times F[n]}(X,E)$, induced by the inclusion
$e^0:F[0]\hra F[n]$, is a trivial fibration.

(b) The map
\[
q:\C{Map}_{Y\times F[n]}(X\times F[n],E)\to \C{Map}_{Y\times
F[n]}(X\times (e^0 F[1]\sqcup_{e^1 F[0]} e^1 F[n-1]),E), \]
induced by the inclusion $e^0F[1]\sqcup_{e^1 F[0]}e^1 F[n-1]\hra
F[n]$, is a trivial fibration.
\end{Lem}

\begin{proof}
(a) Since $E\to Y\times F[n]$ is a left fibration, the map $E^X\to
(Y\times F[n])^X$ is a left fibration by \rl{left} (b). Then by
\rl{leftcart} (b), the map
\[
E^{X\times F[n]}\to E^X\times_{(Y\times F[n])^X}(Y\times
F[n])^{X\times F[n]}
\]
is a trivial fibration. Taking fiber over $f\times\Id_{F[n]}\in
(Y\times F[n])^{X\times F[n]}$, we get the assertion.

(b) Since $e^0 F[1]\sqcup_{e^1 F[0]}e^1 F[n-1]\hra F[n]$ is a
cofibration, the map $q$ is a fibration. Thus it remains to show
that $q$ is a weak equivalence. Consider map
\[
r:\C{Map}_{Y\times F[n]}(X\times (e^0F[1]\sqcup_{e^1F[0]}e^1
F[n-1]),E)\to \C{Map}_{Y\times F[n]}(X,E),
\]
induced by the inclusion $F[0]\overset{e^0}{\lra} e^0 F[1]\hra e^0
F[1]\sqcup_{e^1 F[0]}e^1 F[n-1]$. Since $r\circ q=p$, it is a weak
equivalence by (a). Thus it remains to show that $r$ is a weak
equivalence. But $r$ can be written as a composition of
\[
\C{Map}_{Y\times F[n]}(X\times (e^0F[1]\sqcup_{e^1 F[0]}e^1
F[n-1]),E)\to \C{Map}_{Y\times F[n]}(X\times e^0F[1],E)
\]
and $(e^0)^*:\C{Map}_{Y\times F[n]}(X\times e^0 F[1],E)\to \C{Map}_{Y\times
F[n]}(X,E)$, so it remains to show that both maps are trivial
fibrations. Since the first map is the pullback of the morphism
$\C{Map}_{Y\times F[n]}(X\times e^1 F[n-1],E)\to \C{Map}_{Y\times
F[n]}(X\times e^1 F[0],E)$, both maps are trivial fibrations by
(a).
\end{proof}

\subsection{Proofs} In this subsection we prove Propositions \ref{P:spaces}, 
\ref{P:sn} and \rco{presh}. The most difficult part is \rp{sn} (b), whose proof 
is carried out in \re{step1}-\re{step2} and \re{step4}-\re{step6}.  
We denote $\S_{\C{U}}$ simply by $\S$.

\begin{Emp} \label{E:pfspacesa}
{\bf Proof of \rp{spaces} (a).} By \re{rfib} (a), we have to show
that for every trivial cofibration $i:A\to B$ in $sSp$ with
$B=\Box[n,m]$ the morphism $i^*:\Hom(B,\S)\to\Hom(A,\S)$
is surjective. By \rl{univ} this means that every left fibration
$Y_A\to A$ in $(sSp/A)_{\C{U}}$ extends to a left fibration
$Y_B\to B$ in $(sSp/B)_{\C{U}}$.

Note that composition $j:Y_A\to A\to B$ belongs to
$(sSp/B)_{\C{U}}$, and $B\in sSp_{|\C{U}|}$. Then $j$
decomposes as a composition of a trivial cofibration $Y_A\to X_B$
and a fibration $X_B\to B$ in $(sSp/B)_{\C{U}}$ (by \rl{fibrepl} (a)).
Now let $Y_B\subset X_B$ be the largest simplicial subspace such
that $Y_B\times_B A=Y_A$ (see \rl{fib}). Then $Y_B\to B$ belongs
to $(sSp/B)_{\C{U}}$, and we claim that $Y_B\to B$ is a left
fibration.

Let $X_A:=X_B\times_A B$. Since $A\to B$ is a weak equivalence and
the Reedy model structure is proper, the inclusion $X_A\to X_B$ is
a weak equivalence. Since $Y_A\to X_A\to X_B$ is also a weak
equivalence, we conclude that $Y_A\to X_A$ is a weak equivalence.
Then by \rl{fib}, the projection $Y_B\to B$ is a fibration, while
$Y_A\hra Y_B$ is a weak equivalence. Since $Y_A\to A$ is a left
fibration, $Y_B\to B$ is a left fibration by \rl{lqf}.
\end{Emp}


\begin{Emp} \label{E:pfspacesb}
{\bf Proof of \rp{spaces} (b).}
To show that $\S^{(n)}\to \S^{n+1}$ is a fibration we have to
show that for every trivial cofibration $A\hra B$, the map
$(\S^{(n)})^B\to (\S^{n+1})^B\times_{(\S^{n+1})^A}(\S^{(n)})^A$ is
surjective. Using the observation of \re{remsn},  we have to show
that for every $(n+1)$-tuple of left fibrations
$E^{(0)},\ldots,E^{(n)}$ of over $B$, every diagram
$E^{(0)}|_A\to\ldots\to E^{(n)}|_A$ over $A$ extends to a diagram
$E^{(0)}|_B\to\ldots\to E^{(n)}|_B$ over $B$. For this enough to
show that each restriction map
\[
\Hom_B(E^{(i)},E^{(i+1)})\to
\Hom_B(E^{(i)}|_A,E^{(i+1)})=\Hom_A(E^{(i)}|_A,E^{(i+1)}|_A)
\]
is surjective. Since $E^{(i)}\to B$ is a fibration, $A\to B$ is a
trivial cofibration, and the Reedy model structure is proper, we
get that $E^{(i)}|_A\to E^{(i)}$ is a trivial cofibration. Thus
the assertion follows from the fact $E^{(i+1)}\to B$ is a
fibration.

To show that $\S^{(we)}\to \S^2$ is a fibration, we argue as above
word-by-word, and note that since $E^{(i)}|_A\to E^{(i)}$ are
trivial cofibrations, it follows from 2-out-of-3 that the morphism
$E^{(0)}\to E^{(1)}$ is a weak equivalence if and only if its
restriction $E^{(0)}|_A\to E^{(1)}|_A$ is a weak equivalence.
\end{Emp}

\begin{Emp} \label{E:pfspacesc}
{\bf Proof of \rp{spaces} (c).} We have to show that for every cofibration $A\hra B$, the map 
$(\S^{(we)})^B\to \S^B\times_{\S^A}(\S^{(we)})^A$ is
surjective. Let $E^{(0)}\to B$ belong to $(LFib/B)_{\C{U}}$, $E_A^{(1)}\to A$
belong to $(LFib/A)_{\C{U}}$, and $\phi':E^{(0)}|_A\to E_A^{(1)}$ be 
a weak equivalence over $A$. We have to show that $\phi'$
extends to a weak equivalence $\phi:E^{(0)}\to E^{(1)}$ over $B$
such that $E^{(1)}\to B$ belongs to $(LFib/B)_{\C{U}}$.

By \rl{pushout}, the pushout $E^{(0)}\sqcup_{E^{(0)}|_A} E_A^{(1)}$
is a quasifibration over $B$, whose restriction to $A$ is a
fibration $E_A^{(1)}\to A$. Therefore, by \rco{almost}, there
exists a fibrant replacement $E^{(1)}\to B$ in $(sSp/B)_{\C{U}}$
such that $E^{(1)}|_A=E_A^{(1)}$.

By construction, $\phi'$ extends to a morphism
$\phi:E^{(0)}\overset{\phi_1}{\lra} E^{(0)}\sqcup_{E^{(0)}|_A}
E_A^{(1)}\overset{\phi_2}{\lra} E^{(1)}$ of fibrations over $B$.
Moreover, since $\phi'$ is a weak equivalence, its pushout
$\phi_1$ is a weak equivalence, hence $\phi$ is a weak equivalence
as well. Since $E^{(0)}\to B$ is a left fibration, $E^{(1)}\to B$ is a left
fibration by \rl{lqf}.
\end{Emp}

\begin{Emp} \label{E:pfspacesd}
{\bf Proof of \rp{spaces} (d).} Using observations of \re{remsn}
(b), it remains to show that for every $a\in\S^{(we)}$,
$b\in\S^{(1)}$ such that $a\sim b$ in $\S^{(1)}$, we have $b\in
\S^{(we)}$. Since $\S^{(we)}\to \S\times \S$ is a fibration, while
$\S^{(1)}$ is fibrant (by \rp{spaces} (a),(b)), we may assume that
$a\sim b$ in some fiber of $\S^{(1)}\to\S\times\S$.

Then $a$ corresponds to a weak equivalence $\phi_a:E^{(0)}\to
E^{(1)}$, $b$ corresponds to a morphism $\phi_b:E^{(0)}\to
E^{(1)}$, assumption $a\sim b$ means that the maps $\phi_a$ and
$\phi_b$ are homotopic. Then $b$ is a weak equivalence, hence
$b\in\S^{(we)}$.
\end{Emp}

\begin{Emp} \label{E:pfsna}
{\bf Proof of \rp{sn} (a).} By \rl{univ}, for every $K\in sSp$ the
set $\Hom(K, \S^{\Dt[1]})= \Hom(K\times\Dt[1], \S)$ is in
bijection with the set of left fibrations $E\to K\times\Dt[1]$ in
$(sSp/K\times\Dt[1])_{\C{U}}$, while the set $\Hom(K, \S^2)$ is in
bijection with the set of pairs of left fibrations $E^{(0)}\to
K,E^{(1)}\to K$ in $(sSp/K)_{\C{U}}$. Moreover, the projection
$\S^{(we)}\to\S^2$ sends $E\to K\times\Dt[1]$ to a pair
$E|_0:=E|_{K\times\{0\}}$ and $E|_1:=E|_{K\times\{1\}}$, and the
set $\Hom(K, \S^{(we)})$ is in bijection with the set of weak
equivalences $\phi:E^{(0)}\to E^{(1)}$ of left fibrations over $K$
(by \re{remsn}).

By \rl{heprop} (c) for every left fibration $E\to
K\times\Dt[1]$ there exists a weak equivalence $\phi:E|_0\to E|_1$
of left fibrations over $K$. We take $K:=\S^{\Dt[1]}$ and $E$ be
the left fibration, corresponding to $\Id_K$, then $\phi$ gives
rise to the morphism $\psi:\S^{\Dt[1]}\to \S^{(we)}$ over $\S^2$.

Moreover, $p_0:\S^{(we)}\to\S$ is a  trivial fibration by \rp{spaces}
(c), while $\dt_0:\S^{\Dt[1]}\to\S$ is a trivial fibration, because 
$\S$ is fibrant (see \rp{spaces} (a)), and $\Dt[0]\to\Dt[1]$ is
a trivial cofibration. Then $\psi$ is a weak equivalence by
2-out-of-3, hence a homotopy equivalence by \rl{heprop} (b).
\end{Emp}

\begin{Emp} \label{E:pfpresh}
{\bf Proof of \rco{presh}.} The homotopy equivalence
$\S^{(we)}\to \S^{\Dt[1]}$ over $\S^2$
from \rp{sn} (a) induces a homotopy equivalence
\[
\Map(K,\S^{(we)})\to\Map(K,\S^{\Dt[1]})=
\Map(\Dt[1],\S^K)
\]
over $\Map(K,\S^2)$, hence a homotopy equivalence
between fibers over $(\al,\beta)\in\Map(K,\S^2)$. Since
fiber $\Map(K,\S^{(we)})_{\al,\beta}\neq\emptyset$ means
that $E_{\al}$ and $E_{\beta}$ are homotopy equivalent over $K$
(by \re{remsn}(b)), while
$\Map(K,\S^{\Dt[1]})_{\al,\beta}\neq\emptyset$ means
that $\al\sim\beta$ in $\S^K$, we get the assertion.
\end{Emp}

\begin{Emp} \label{E:step1}
{\bf Construction of $\psi^{(n)}$.} By \re{remsn}, the identity
map $\Id_{\S^{(n)}}$ corresponds to a diagram
$\phi:E^{(0)}\overset{\phi_1}{\lra} \ldots \overset{\phi_n}{\lra}
E^{(n)}$ of left fibrations over $\S^{(n)}$ in $(sSp/\S^{(n)})_{\C{U}}$. To
define a morphism $\psi^{(n)}:\S^{(n)}\to \S^{F[n]}$ over
$\S^{n+1}$, we have to construct a left fibration $E\to
\S^{(n)}\times F[n]$ in $(sSp/\S^{(n)}\times F[n])_{\C{U}}$ such
that $E|_i:=E|_{\S^{(n)}\times\{i\}}=E^{(i)}$.

Consider iterated discrete cylinder $p:Cyl^d(\phi)\to
\S^{(n)}\times F[n]$ (see \re{itercyl}). Then $p$ is a left
quasifibration in $(sSp/\S^{(n)}\times F[n])_{\C{U}}$ (see
\rl{disc}) such that the restriction $Cyl^d(\phi)|_i=E^{(i)}$ is a
fibration over $\S^{(n)}$. Thus, by \rco{almost}, there exists a
fibrant replacement $p':E\to \S^{(n)}\times F[n]$ of $p$ in
$(sSp/\S^{(n)}\times F[n])_{\C{U}}$ such that $E|_i=
Cyl^d(\phi)|_i$ for all $i$. Then $p'$ is a left fibration 
by \rco{lqf}.
\end{Emp}

\begin{Emp} \label{E:step1'}
{\bf Uniqueness of $\psi^{(n)}$.} Notice that if $E'$ is another
fibrant replacement of $Cyl^d(\phi)\to \S^{(n)}\times F[n]$ such
that $E'|_i= Cyl^d(\phi)|_i$ for all $i$, then there exists a weak
equivalence $E\to E'$ over $\S^{(n)}\times F[n]$, which is
identity over each $\{i\}\in F[n]$. Now it follows from
\rco{presh} that the morphism $\psi'^{(n)}:\S^{(n)}\to
\S^{F[n]}$, corresponding to $E'$, is homotopic to $\psi^{(n)}$
over $\S^{F[n]}$.
\end{Emp}

\begin{Emp} \label{E:step2}
{\bf Modular interpretation of $\psi^{(n)}$}. Note that a morphism
$\varphi:K\to \S^{(n)}$ corresponds to the diagram
$\varphi^*(\phi)$ of left fibrations over $K$, while the
composition $\psi^{(n)}\circ \varphi:K\to \S^{F[n]}$ corresponds
to the left fibration $\varphi^*(E)$. Recall that $Cyl^d(\phi)\to
K\times F[n]$ is a quasifibration and that $E$ is a fibrant
replacement of $Cyl^d(\phi)$ such that $E|_i=E^{(i)}$ for all $i$.
Hence it follows from \rco{fibrep} that $\varphi^*(E)$ is a
fibrant replacement of
$\varphi^*(Cyl^d(\phi))=Cyl^d(\varphi^*(\phi))$ such that
$\varphi^*(E)|_i=\varphi^*(E^{(i)})$ for all $i$.
\end{Emp}

\begin{Emp} \label{E:pfsnc}
{\bf Proof of \rp{sn} (c).} By \re{remsn}, the composition 
$\mu^*\circ \psi^{(n)}:\S^{(n)}\to \S^{F[m]}$ corresponds to the left fibration
$\mu^*(E)\to \S^{(n)}\times F[m]$, which is as in \re{step2} is a fibrant 
replacement of $\mu^*(Cyl^d(\phi))=Cyl^d(\mu^*(\phi))$. Similarly, $\psi^{(m)}\circ\mu^*:\S^{(n)}\to \S^{F[m]}$
also corresponds to a fibrant replacement of $Cyl^d(\mu^*(\phi))$. Since all fibrant
replacement are weakly equivalent, two compositions are homotopic
by \rco{presh}.
\end{Emp}

It remains to show that $\psi=\psi^{(n)}$ is a homotopy equivalence over $\S^{n+1}$. 

\begin{Emp} \label{E:step4}
{\bf  Reduction.} By \rl{heprop} (a), we have to show that  
for every map $\eta:M\to \S^{n+1}$, the  map
$\pi_0(\psi/\eta):\pi_0(\Map_{\S^{n+1}}(M,\S^{(n)}))\to
\pi_0(\Map_{\S^{n+1}}(M,\S^{F[n]}))$, induced by $\psi$, 
is a bijection.

Let $\eta$ corresponds to an $(n+1)$-tuple $H^{(0)},\ldots,
H^{(n)}$ of left fibrations over $M$. Then $\varphi\in
\Hom_{\S^{n+1}}(M,\S^{(n)})$ corresponds to  diagrams
$\varphi:H^{(0)}\overset{\varphi_1}{\lra}\ldots\overset{\varphi_n}{\lra}
H^{(n)}$ over $M$, and $\tau\in \Hom_{\S^{n+1}}(M,\S^{F[n]})$
corresponds to left fibrations $H\to M\times F[n]$ such that
$H|_i=H^{(i)}$ for all $i$.

Using \rco{presh} (a) and \re{step2}, we see that
$\psi\circ\varphi\sim \tau$ in $\Map_{\S^{n+1}}(M,\S^{F[n]})$  if
and only if there exists a weak equivalence $\nu:Cyl^d(\varphi)\to
H$ over $M\times F[n]$ such that  $\nu|_i:H^{(i)}\to H|_i=H^{(i)}$
is the identity. Moreover, by \re{fn} (e), this happens if and only if 
there exists a morphism $\nu:Cyl^d(\varphi)\to H$ over $M\times F[n]$ such that
$\nu|_i:H^{(i)}\to H|_i=H^{(i)}$ is $\Id_{H^{(n)}}$ for all $i$.
\end{Emp}

\begin{Emp} \label{E:step5}
{\bf Proof of surjectivity of $\pi_0(\psi/\eta)$.} We have to show
that for every left fibration $H\to M\times F[n]$ such that
$H|_i=H^{(i)}$ for all $i$ there exists a diagram $\varphi$ and a
morphism $\nu:Cyl^d(\varphi)\to H$ over $M\times F[n]$ such that
each $\nu|_i:H^{(i)}\to H|_i=H^{(i)}$ is the identity. We
construct $\varphi$ and $\nu$ by induction on $n$. If $n=0$, then
$\varphi$ is empty, $Cyl^d(\varphi)=H^{(0)}=H$, so $\nu=\Id_H$
does the job.

Assume that $n>0$. By induction hypothesis, there exists a diagram
\\ $\varphi^{(1)}:H^{(1)}\overset{\varphi_2}{\lra}\ldots
\overset{\varphi_n}{\lra} H^{(n)}$ over $M$ and a morphism $e^1
Cyl^d(\varphi(1))\to H|_{e^1 F[n-1]}\subset H$ over $M\times F[n]$
such that $\nu|_i:H^{(i)}\to H|_i=H^{(i)}$ is the identity for all
$i>0$. In particular, we have a morphism $\nu[1]:H^{(1)}\times
e^1 F[n-1]\to e^1 Cyl^d(\varphi(1))\to H$ over $M\times F[n]$ such that
$\nu[1]|_1=\Id_{H^{(1)}}$.

Since $Cyl^d(\varphi)=(H^{(0)}\times F[n])\sqcup_{(H^{(0)}\times
e^1 F[n-1])}e^1 Cyl^d(\varphi(1))$, it remains to construct a
morphism $\varphi_1:H^{(0)}\to H^{(1)}$ over $M$ and a morphism
$\nu[0]:H^{(0)}\times F[n]\to H$ over $M\times F[n]$ such that
$\nu[0]|_0=\Id_{H^{(0)}}$, and restriction $\nu[0]|_{e^1 F[n-1]}$
decomposes as a composition $H^{(0)}\times e^1
F[n-1]\overset{\varphi_1}{\lra} H^{(1)}\times e^1
F[n-1]\overset{\nu[1]}{\lra}H$.

Since $H\to M\times F[n]$ is a left fibration, the inclusion
$H^{(0)}=H|_0\hra H$ extends to a morphism $\nu'[0]:H^{(0)}\times
e^0 F[1]\to H|_{e^0 F[1]}\subset H$ over $M\times F[n]$ (see
\rl{lfibr} (a)). Denote $\nu'[0]|_1:H^{(0)}\to H^{(1)}$ by
$\varphi_1$, and define $\nu''[0]:H^{(0)}\times e^1 F[n-1]\to
H|_{e^1 F[n-1]}\subset H$ to be the composition
$\nu[1]\circ\varphi_1$. Then $\nu'[0]$ and $\nu''[0]$ define a
morphism $H^{(0)}\times (e^0 F[1]\sqcup_{e^1 F[0]}e^1 F[n-1])\to H$
over $M\times F[n]$, which by \rl{lfibr} (b) can
be extended to all of $H^{(0)}\times F[n]$. 
\end{Emp}

\begin{Emp} \label{E:step6}
{\bf Proof of injectivity of $\pi_0(\psi/\eta)$.} Fix a left
fibration $H\to M\times F[n]$, and consider all diagrams
$\varphi:H^{(0)}\overset{\varphi_1}{\lra}\ldots
\overset{\varphi_n}{\lra} H^{(n)}$ over $M$ for which there exists
a morphism $\nu:Cyl^d(\varphi)\to H$ over $M\times F[n]$ such that
each $\nu|_j$ is the identity. We have to show that each 
$\pi_0(\varphi_j)\in\pi_0(\Map_{M}(H^{(j)},H^{(j+1)}))$
only depends on $H$.

Consider canonical embedding $\iota_j:H^{(j)}\times e^j F[n-j]\to
Cyl^d(\varphi)$ (see \re{itercyl}). Then the composition
$\nu\circ\iota_j:H^{(j)}\times e^j F[n-j]\to H$ is such that 
$(\nu\circ\iota)|_j:H^{(j)}\to H^{(j)}$ is $\Id_{H^{(j)}}$, while
$(\nu\circ\iota)|_{j+1}:H^{(j)}\to H^{(j+1)}$ is $\varphi_j$. Thus
it remains to show that each 
$\pi_0(\nu\circ\iota_j)\in\pi_0(\Map_{M\times F[n]}(H^{(j)}\times e^j F[n-j],
H))$ only depends on $H$. Since the restriction map $\Map_{M\times
F[n]}(H^{(j)}\times e^j F[n-j], H)\to\Map_{M\times
F[n]}(H^{(j)}\times\{j\}, H)$ is a trivial fibration (by \rl{lfibr} (b)), while
$(\nu\circ\iota_j)|_j=\Id_{H^{(j)}}$, the assertion follows from
\re{pi0} (b).
\end{Emp}


\end{document}